\documentclass{article}

\usepackage{amsmath}
\usepackage{amsthm}
\usepackage{amssymb}
\usepackage{amscd}

\theoremstyle{plain}
\newtheorem{theo}{Theorem}[subsection]
\newtheorem{prop}[theo]{Proposition}
\newtheorem{lemm}[theo]{Lemma}
\newtheorem{cor}[theo]{Cororally}

\theoremstyle{definition}
\newtheorem{defi}[theo]{Definition}
\newtheorem{exa}[theo]{Example}

\numberwithin{equation}{section}

\newcommand{\semiring}{\mathrm{semiring}}
\newcommand{\lambdasemiring}{$\lambda$-semiring}
\newcommand{\prelambdaring}{pre-$\lambda$-ring}
\newcommand{\lambdaf}[1]{$\lambda$-#1}
\newcommand{\Map}[2]{\mathrm{Map}(#1,#2)}
\newcommand{\MapC}[1]{\mathrm{Map}(#1,\mathbb{C})}
\newcommand{\Ztorsionfree}{\mathop{\mathbb{Z}\text{-}\mathrm{torsion}\text{-}\mathrm{free}}\nolimits}
\newcommand{\Qalg}{\mathbb{Q}\text{-}\mathrm{algebra}}
\newcommand{\inn}[3]{\langle #1,#2\rangle_{#3}}

\begin{document}

\author{Tomoyuki Tamura}

\title{Irreducible decomposition and calculating of multiplicity of the symmetric and exterior powers representation of finite groups}



\date{\today}

\maketitle

\thispagestyle{empty}
\begin{abstract}

In this paper\footnote[1]{Keywords and Phrases; Symmetric powers representation; exterior powers representation; finite group; character table; $\lambda$-ring.} we consider symmetric powers representation and exterior powers representation of finite groups, which generated by the representation which has finite dimension over the complex field. We calculate the multiplicity of irreducible component of two representations of some representation by using a character theory of representation and a \prelambdaring, for example, the regular representation.
\end{abstract}

\begin{flushleft}
\Large Introduction
\end{flushleft}

Let $G$ be a finite group. For any representation $\rho :G\rightarrow GL(V)$ where $V$ is a finite dimension vector space over the complex field $\mathbb{C}$, we obtain the representation whose representation space is a $i$-th symmetric power space $S^i(V)$ (it is called to the $i$-th symmetric power representation) and the representation whose representation space is a $i$-th exterior power space $\bigwedge^i(V)$ (it is called to the $i$-th exterior power representation) for any nonnegative integers $i\geq 0$ as follows.
\begin{eqnarray*}
S^i\rho&:&G\rightarrow GL(S^i(V)),\ S^i(\rho)(g)(v_1 \cdots v_i):=(\rho(g)v_1)\cdots(\rho(g)v_i) \\
\Lambda^i\rho&:&G\rightarrow GL(\textstyle\bigwedge^i(V)),\ \Lambda^i\rho(g)(v_1\wedge\cdots\wedge v_i):=\rho(g)v_1\wedge\cdots\wedge \rho(g)v_i.
\end{eqnarray*}
for any $g\in G$ and $v_1,\dots v_i\in V$. Let $V_1,V_2,\dots,V_k$ be all of irreducible representation spaces over $\mathbb{C}$ of $G$. Since the representation $S^i(V)$ and $\bigwedge^i(V)$ are the representation of $G$ and over $\mathbb{C}$, so these representations are completely reducible from Maschkes Theorem and isomorphic to direct sum of some $V_1,V_2,\dots,V_k$ as the representation of $G$. In addition,  representations $S^i(V)$ and $\bigwedge^i(V)$ have a finite dimension over $\mathbb{C}$ too since $V$ has a finite dimension over $\mathbb{C}$. Hence, these representation are isomorphic to direct sum of irreducible representation, and every irreducible component appears finitely, and uniquely. On the other words, we have
\begin{eqnarray*}
S^i(V)\cong \overbrace{V_1 \oplus \cdots \oplus V_1}^{m_{i,1}}\oplus \overbrace{V_2 \oplus \cdots \oplus V_2}^{m_{i,2}}\oplus\cdots\oplus\overbrace{V_k \oplus \cdots \oplus V_k}^{m_{i,k}} \\
\textstyle\bigwedge^i(V)\cong \overbrace{V_1 \oplus \cdots \oplus V_1}^{n_{i,1}}\oplus\overbrace{V_2 \oplus \cdots \oplus V_2}^{n_{i,2}}\oplus\cdots\oplus\overbrace{V_k \oplus \cdots \oplus V_k}^{n_{i,k}}
\end{eqnarray*}
and nonnegative integers  $m_{i,1},m_{i,2},\dots,m_{i,k},n_{i,1},n_{i,2},\dots,n_{i,k}$ are determined finitely and uniquely. The purpose of this paper is that calculating about it.

We consider about this character $\chi$ since $V$ has a finite dimension. A set of all maps from $G$ to $\mathbb{C}$, which is said to be $\MapC{G}$, becomes to \prelambdaring (In fact, it is $\lambda$-ring). A character $\chi$ belongs to $\MapC{G}$. In addition, $S^i(\chi)$, where maps $S^i$ are called to symmetric powers operation, implies that the character of the $i$-th symmetric powers representation, and $\lambda^i(\chi)$, where maps $\lambda^i$ is called to $\lambda$-operation, implies that the character of the $i$-th exterior powers representation. we calculate $S^i(\chi)$ and $\lambda^i(\chi)$, and be expressed by sum of irreducible character.

Actually the general method of calculation of multiplicity using of character theory is known in the \cite{bu-1}. In this paper, we describe the how to calculate more efficient on the method of \cite{bu-1}  about some finite groups and some representation.

 The structure of the paper, $\S1$ to $\S4$ is a preparation, which is known as very well. It is discussed more detail than this paper in some reference. Main result of this paper describe in $\S\ref{calculateresult1}$ and $\S\ref{somecalculate}$, however  about $\S\ref{calculateresultoperate}$ is described about the method which is written in  \cite{bu-1}.  $\S\ref{semiringsection}$ gives the definition and properties of a $\semiring$ and the ring completion. In $\S\ref{logdiffsection}$, we provide a theory of symmetric polynomial and logarithmic derivation. In $\S\ref{lambdasection}$, we descibe the definition and properties of \prelambdaring At last, we mix $\S1, \S2$ and $\S3$ in $\S\ref{charactersection}$, and consider the relation of characters and $\lambda$-operation.

 In $\S\ref{calculateresultgenerate}$, we describe the method of generating function of multiplicity. In $\S\ref{calculateresult1st}$ to $\S\ref{calculateresultcentral}$, we discuss about some representation space, which are one-dimension representation in   $\S\ref{calculateresult1st}$, a regular representation and that generated by action $G$ to the quotient $G/N$ where $N$ is a normal subgroup of $G$, in  $\S\ref{calculateresultburnside}$ and made by one-dimension representation of $N$ that $N$ is contained in a center of $G$ in $\S\ref{calculateresultcentral}$.\ In $\S\ref{calculateresultirrirr}$, we describe the relation of the irreducible and reducible representation. In $\S\ref{calculateresultquotient}$, we introduce the method which is using the quotient group of $G$. At last we calculate the irreducible representation which has dimension more than 2 in some finite group in $\S\ref{somecalculate}$.

\tableofcontents

\thispagestyle{empty}
\newpage

\setcounter{page}{1}

\section{Semiring}\label{semiringsection}
First, we describe about semiring, and the group and ring completion. we will be used it a set of isomorphism class of representation and its character. We can see more about semiring in $\S9.3$ of \cite{bu-4} in detail.

In this paper , Assume that any ring has a unit element 1.

\subsection{Semiring}

\begin{defi}
We call that a semiring is a triple $(S, ``+", ``\cdot")$, where $S$ is a set, $``+"$ is the additive operation, $``\cdot"$ is the multiplicative operation, and it is satisfied all the axioms of a ring except the existence of a negative or additive inverse.
\end{defi}

\begin{defi}
A morphism from a semiring $S$ to a semiring $S'$ is the function $f:S\rightarrow S'$ such that $f(r+s)=f(r)+f(s), f(rs)=f(r)f(s), f(0)=0, f(1)=1$ for any elemets $r,s$ in $S$.
\end{defi} 

\subsection{The ring completion}\label{semiringcompletion}
Next, we describe the ring completion of a semiring.

\begin{defi}
The ring completion of a semiring $S$ is a pair $(S^*, \theta)$, where $S^*$ is a ring and $\theta$ is a morphism from a semiring $S$ to a ring $S^*$ such that if $f:S\rightarrow R$ is a morphism of semiring, there exists a ring morphism $g:S^*\rightarrow R$ such that $g\circ\theta=f$. Moreover, $g$ is required to be unique.
\end{defi}

In [1], Husemoller describe that there is a ring completion of any semiring $S$. We consider the set $S\times S$, and put the following equivalence relation. we define that two elements $(a,b)$ and $(a',b)$ in $S\times S$ are equivalent if there is a element $c\in S$ such that $a+b'+c=a'+b+c$ holds. Let $\inn{a}{b}{}$ denote the equivalent class of $(a,b)$, and let $S^*$ denote the set of all equivalent classes $\inn{a}{b}{}$.
$S^*$ becomes a ring with operators $\inn{a}{b}{}+\inn{c}{d}{}:=\inn{a+c}{b+d}{},\ \inn{a}{b}{}\inn{c}{d}{}=\inn{ac+bd}{ad+bc}{}$, which is well-defined.  

\begin{prop}\label{semiuniv}
\begin{enumerate}
\item The function $\theta:S\rightarrow S^*$ by $\theta(a)=\inn{a}{0}{}$ is a morphism of a semiring.
\item (The group completion) Let $A$ be an abelian with the additive operator $``+"$, and let $f:(S,``+")\rightarrow (A,``+")$ be a morphism of monoid. Then, there exists a group homomorphism uniquely $g:(S^*, ``+")\rightarrow (A,``+")$ such that $g\circ\theta=f$.  
\item (The ring completion) Let $R$ be a ring, and let $f: S \rightarrow A $ be a morphism of semiring. Then, there exists a ring homomorphism uniquely $g:S^*\rightarrow A$ such that $g\circ\theta=f$.  
\end{enumerate}
\end{prop}
A map $g$ of this Proposition is constructed to $g(\inn{a}{b}{})=f(a)-f(b)$.

From this Proposition, we can see that there exists a ring completion of any semirings.

\newpage

\section{Polynomial and formal power series}\label{logdiffsection}
The purpose of this chapter, we consider a polynomial, a formal power series, a symmetric polynomial, and those properties and logarithmic derivative, which will be used in next chapter, $\lambda$-ring. Let $R$ be a commutative ring. We call $R[x_1,\dots,x_n]$ that a polynomial ring over $R$ in $n$ independent variables $x_1,\dots,x_n$, and we call $R[[x_1,\dots,x_n]]$ that a set of formal power series over $R$ in $n$ independent variables $x_1,\dots,x_n$, for each $n\in\mathbb{N}$. 
we can see about symmetric polynomial in \cite{bu-7}, and we can see about logarithmic derivative in $\S3$ of \cite{bu-9}.

\subsection{Derivative of a formal power series}\label{logdiffsection1}
Let $R[[t]]$ be a set of formal power series over $R$ in variable $t$. we define a derivation of a formal power series $f(t)=\displaystyle\sum_{i=0}^{\infty}a_it^i$ in $R[[t]]$  by
\[ \dfrac{d}{dt}f(t):=\sum_{i=1}^{\infty}ia_it^{i-1}.\]
This Definition is different of the Definition in the real functions, which is using a property of limit. However, we can also use next Proposition in this Definition.

\begin{prop}\label{diff}
The follow statements hold for elements $a$, $b$ in $R$ and $f$, $g$ in $R[[t]]$.
\begin{enumerate}
\item $\dfrac{d}{dt}(af+bg)(t)=a\dfrac{d}{dt}f(t)+b\dfrac{d}{dt}g(t).$
\item $\dfrac{d}{dt}(fg)(t)=f(t)\Big(\dfrac{d}{dt}g(t)\Big)+\Big(\dfrac{d}{dt}f(t)\Big)g(t).$
\item If the constant term of $g$ is $0$, one has $\dfrac{d}{dt}f(g(t))=\Big(\dfrac{d}{dt}f\Big)(g(t))\dfrac{d}{dt}g(t)$.
\end{enumerate}
\end{prop}

\subsection{Symmetric polynomial}
Let $S_n$ be a symmetric group each $n=1,2,3\dots$ . For each $\sigma\in S_n$ and polynoimal $f\in\mathbb{Z}[[x_1,\dots,x_n]$, which coefficient are integer, we define $(\sigma f)(x_1,\dots,x_n)$ by $f(x_{\sigma^{-1}(1)},\dots,x_{\sigma^{-1}(n)}$. Also, we call  $\Lambda(X^n)$ as follows, and its element is called a symmetric polynomial.
\[ \Lambda(X^n):=\{\ f\in \mathbb{Z}[x_1,\dots,x_n]\ |\ \sigma f=f\ \mathrm{for}\ {}^{\forall} \sigma \in S_n\} \]

For each $k=1,\dots,n$ , a $k$-th elementary symmetric polynomial $e_k$ on $x_1,\dots,x_n$ is denoted by the following statement.
\begin{eqnarray}\label{element}
e_k:=\sum_{1\leqq i_1<\cdots<i_k\leqq n}x_{i_1}\cdots x_{i_k}.
\end{eqnarray}
An elementary symmetric polynomial belongs to $\Lambda(X^n)$, and the elementary symmetric polynomials are algebraically independent. Hence, we can consider the following  evaluation map of $n$ elements $a_1,\dots,a_n$ in $R$. It is a ring homomorphism.
\begin{eqnarray}\label{evalue}
F:\Lambda(X^n)=\mathbb{Z}[e_1,\dots,e_n]\rightarrow R, F(e_j)=a_j\ (j=1,\dots,n). 
\end{eqnarray}

Next, we denote a symbol of complete symmetric polynomial and power-sum symmetric polynomial. For each $k\in\mathbb{N}$, $h_k(e_1,\dots,e_n)$ is a $k$-th complete symmetric polynomial on $x_1,\dots, x_n$, which is expressed by polynomial of  $e_1,\dots,e_n$. Similarly, $Q_k(e_1,\dots,e_n)$ is a $k$-th power-sum symmetric polynomial on $x_1,\dots, x_n$, which is expressed by polynomial of $e_1,\dots e_n$. On the other hand, the following statement holds.

\begin{eqnarray}
\label{hyperelement} h_k(e_1,\dots,e_n)&:=&\sum_{1\leqq i_1\leqq\cdots\leqq i_k\leqq n}x_{i_1}\cdots x_{i_k}, \\
\label{powerselement}Q_k(e_1,\dots,e_n)&:=&x_1^k+\cdots+x_n^k.
\end{eqnarray}
 
\subsection{The inverse of multiplication and the logarithmic derivative}
\begin{defi}
For any ring homomorphism $F:R_1\rightarrow R_2$, we define a map  $F_{\Lambda}:R_1[[t]]\rightarrow R_2[[t]]$ by following equality for any formal power series $\sum_{i=0}^{\infty}a_it^i$.
\begin{eqnarray}\label{lambdaup}
 F_{\Lambda}\Big(\sum_{i=0}^{\infty}a_it^i\Big):=\sum_{i=0}^{\infty}F(a_i)t^i.
\end{eqnarray}
$F_{\Lambda}$ is a $\mathrm{ring\ homomorphism}$ from $R_1[[t]]$ to $R_2[[t]]$.
\end{defi}

In the set of symmetric polynomial $\Lambda(X^n)[[t]]$, By using (\ref{element}), (\ref{hyperelement}) and the generating function of elementary and complete symmetric polynomials, we obtain the following equality.

\begin{eqnarray}
\label{geneelement} \prod_{i=1}^n(1+x_it)&=&1+\sum_{i=1}^ne_it^i,\\
\label{genehyper}\prod_{i=1}^n\Big(\dfrac{1}{1-x_it}\Big)&=&1+\sum_{i=1}^nh_i(e_1,\dots,e_n)t^i.
\end{eqnarray}
Hence, we can see the next Lemma from (\ref{evalue}),\ (\ref{lambdaup}) and (\ref{geneelement}).
\begin{lemm}
In any ring $R$ and $n$ element $a_1,\dots,a_n$ in $R$, we define $F:\Lambda(X^n)\rightarrow R$ be a evaluation map by $F(e_j)=a_j\ (j=0,1,\dots,n)$, so we obtain
\begin{eqnarray}\label{polyevalue}
 F_{\Lambda}\Big( \prod_{i=1}^n(1+x_it)\Big)=1+a_1t+a_2t^2+\cdots+a_nt^n.
\end{eqnarray}
\end{lemm} 
Let $R$ be a commutative ring. We define the strict power series over $R$ by the formal power series with coefficients in $R$ and it has constant term 1. and $\Lambda(R)$ is defined by the set of all strict power series over $R$. For any ring homomorphism $F:R_1\rightarrow R_2$. It is clear that $F_{\Lambda}(\Lambda(R_1))\subset\Lambda(R_2)$ holds.

\cite{bu-1} and \cite{bu-9} called $\Lambda(R)$ the universal-$\lambda$-ring over $R$, which has a $\lambda$-ring structure as the name represents. However, we do not touch about it except of its addition, which is the multiplication of $R[[t]]$. About it, we have the next Proposition.
 
\begin{prop}\label{h-symm}
$\Lambda(R)$ has a structure of group with the multiplication of formal power series over $R$. Also, let $f=1+\sum_{i=1}^{\infty}a_it^i$ be a strict power series over $R$, and let inverse of $f$ be $f^{-1}=1+\sum_{i=1}^{\infty}b_it^i$. Then the coefficient $b_i$ is a polynomial with integer coefficients in $a_1,\dots a_i$. Moreover,  that $b_i=(-1)^i h_i(a_1,\dots,a_i)$ are hold for each $i\geq 1$.
\end{prop}
\begin{proof} It is clear that $\Lambda(R)$ is closed about the multiplication, and has a unit element 1, we consider the inverse of multiplication. Let $ f=1+\sum_{i=1}^{\infty}a_i t^i\in\Lambda(R)$ be a strict power series over $R$. Then we  let $b_n$ be the following equality by induction of $n\in\mathbb{N}$.
\begin{eqnarray}\label{h-symm1}
\begin{split}
b_1 &:= -a_1, \\
b_n &:= -(b_{n-1}a_1+b_{n-2}a_2+\cdots+b_1a_{n-1}+a_n)\ (n\geq 2).
\end{split}
\end{eqnarray}
It follows that $\sum_{p+q=n}b_pa_q =0$, we have $ (1+\sum_{n=1}^{\infty}b_nt^n)f=1$, which implies that we finish to proof of the existence of inverse $ f^{-1}=1+\sum_{i=1}^{\infty}b_it^i$. In particular, $b_n$ is a polynomial with integer coefficients in $a_1,\dots,a_n$ from (\ref{h-symm1}), hence the coefficient of $t^n$ in $g_n=1+a_1t+\cdots+a_n t^n$ is equal to $b_n$. Also, let $F_n:\Lambda(X^n)\rightarrow R$ be a evaluation map by setting\ $F_n(e_j)=a_j\ (j=1,2,\dots,n)$, we have
\begin{eqnarray*}
g_n^{-1}&=&{F_n}_{\Lambda}\Big(\prod_{j=1}^n(1+x_it)\Big)^{-1} \\
 &=& {F_n}_{\Lambda}\Big(\prod_{j=1}^n\dfrac{1}{1-x_i(-t)}\Big) \\
 &=& {F_n}_{\Lambda}\Big(1+\sum_{j=1}^{\infty}h_j(e_1,\dots,e_n)(-t)^j \Big) =1+\sum_{j=1}^{\infty}(-1)^j h_j(a_1,\dots,a_n)t^j
\end{eqnarray*}
from (\ref{genehyper}),\ (\ref{polyevalue}). Comparing of coefficients of $t^n$ in this equality, we have $b_n=(-1)^ih_n(a_1,\dots,a_n)$.
\end{proof}

\begin{defi}\label{logdiff}
We define the map $\dfrac{d}{dt}\log:\Lambda(R)\rightarrow R[[t]]$ by $\dfrac{d}{dt}\log(f):= f'f^{-1}$ for any $f\in \Lambda(R)$.
\end{defi}
Definition \ref{logdiff} is given the result of logarithmic derivation of the real function, then we call this map the logarithmic derivation too. We can define it since there is a inverse element for any elements of $\Lambda(R)$, which is domain of it.

From $\mathrm{Proposition\ \ref{diff}\ (2)}$, the follow statement hold for any $f,g\in\Lambda(R)$.
\begin{eqnarray}\label{logdiffhom}
\dfrac{d}{dt}\log(fg)=\dfrac{d}{dt}\log(f)+\dfrac{d}{dt}\log(g).
\end{eqnarray}
It implies that the logarithmic derivation is a homomorphism from the multiplication of $\Lambda(R)$ to the addition of $R[[t]]$.

\begin{prop}\label{logdifffunc}
For any ring homomorphism $F:R_1\rightarrow R_2$, the logarithmic derivation and a map $F_{\lambda}:\Lambda(R_1)\rightarrow\Lambda(R_2)$ are commutative on $\Lambda(R_1)$. In the other words, The following formula is hold on $\Lambda(R_1)$.
\[ \dfrac{d}{dt}\log\circ F_{\Lambda}=F_{\Lambda}\circ\dfrac{d}{dt}\log .\]
\end{prop}
\begin{proof} From Proposition\ \ref{h-symm}, we have
\begin{eqnarray*}
F_{\Lambda}\Big(\dfrac{d}{dt}\log f(t)\Big)&=&F_{\Lambda}\Big(\Big(\sum_{i=1}^{\infty}ia_it^{i-1}\Big)\Big(1+\sum_{i=1}^{\infty}(-1)^ih_i(a_1,\dots,a_i)t^i\Big)\Big) \\
&=&\Big(\Big(\sum_{i=1}^{\infty}if(a_i)t^{i-1}\Big)\Big(1+\sum_{i=1}^{\infty}(-1)^ih_i(f(a_1),\dots,f(a_i))t^i\Big)\Big) \\
&=&\dfrac{d}{dt}\log(F_{\Lambda}(f)). 
\end{eqnarray*}
for all $f(t)=1+\sum_{i=1}^{\infty}a_it^i\in\Lambda(R_1).$
\end{proof}

Next, we say the relation about the strict power series and what is mapping by the logarithmic derivation.

\begin{lemm}\label{logdiffchalemma}
Let $a_1,a_2,\dots$ be the series of elements of $R$, and let $f_n$ be equal to  $1+a_1t+a_2t^2+\cdots+a_nt^n\in\Lambda(R)$ for each $n\in\mathbb{N}$. Then, the followig statements hold.
\begin{enumerate}
\item For integer $n\in\mathbb{N}$, we have \ $-t\dfrac{d}{dt}\log f_n(t)=\displaystyle\sum_{i=1}^{\infty}(-1)^iQ_i(a_1,\dots,a_n)t^i$.
\item For integers $k,n\in\mathbb{N}$ with $k\leq n$, we have $Q_k(a_1,\dots,a_n)=Q_k(a_1,\dots,a_k)$.
\item For integers $k,n\in\mathbb{N}$ with $k>n$,\ we have $Q_k(a_1,\dots,a_n)\\=Q_k(a_1,\dots,a_n,0,\dots,0)$ , where the numbers of  0 of the right hand side is $k-n$.
\end{enumerate}
\end{lemm}
\begin{proof}
$(1)$ We consider a evaluation map $F_n:\Lambda(X^n)\rightarrow R$ setting by $F_n(e_j)=a_j \ (j=1,2,\dots,n)$. From (\ref{powerselement}),\ (\ref{polyevalue}) and  Proposition\ \ref{logdifffunc} , we have
\begin{eqnarray*}
-t\dfrac{d}{dt}\log f_n(t)&=&-t\dfrac{d}{dt}\log {F_n}_{\Lambda}\Big(\prod_{j=1}^n (1+x_it) \Big) \\
&=&{F_n}_{\Lambda}\Big(-t\dfrac{d}{dt}\log\prod_{j=1}^n(1+x_it)\Big) \\
&=&{F_n}_{\Lambda}\Big(-t\sum_{j=1}^n\dfrac{d}{dt}\log(1+x_it)\Big) \\
&=&{F_n}_{\Lambda}\Big(-t\sum_{j=1}^n\sum_{i=0}^{\infty}x_j^{i+1}(-t)^i \Big) \\ 
&=&{F_n}_{\Lambda}\Big(\sum_{i=1}^{\infty}(-1)^iQ_i(e_1,\dots,e_n)t^i\Big) =\sum_{i=1}^{\infty}(-1)^iQ_i(a_1,\dots,a_n)t^i.
\end{eqnarray*}

$(2)$ The coefficient of $t^k$ in $-t\dfrac{d}{dt}\log f_n(t)$ is expressed by the polynomial with integer coefficients in $a_1,\dots,a_k$, so it is equal to the coefficient of $t^k$ in $-t\dfrac{d}{dt}\log f_k(t)$. so, we have $Q_k(a_1,\dots,a_k)=Q_k(a_1,\dots,a_n).$ from $(1)$.

$(3)$ The polynomial $f_n$ can be seen as $f_n=1+a_1t+a_2t^2+\cdots+a_nt^n+0t^{n+1}+\cdots+0t^k$. Hence we have
\begin{eqnarray}\label{logdiffchalemma1}
-t\dfrac{d}{dt}\log f_n(t)=\sum_{i=1}^{\infty}(-1)^iQ_i(a_1,\dots,a_n,0,\dots,0)t^i
\end{eqnarray}
from $(1)$. On the other hand, we have
\begin{eqnarray}\label{logdiffchalemma2}
-t\dfrac{d}{dt}\log f_n(t)=\sum_{i=1}^{\infty}(-1)^iQ_i(a_1,\dots,a_n)t^i
\end{eqnarray}
from $(1)$. Therefore, The equalities $Q_k(a_1,\dots,a_n)=Q_k(a_1,\dots,a_n,0,\dots,0)$ is given by comparing of the coefficient of $t^k$ in (\ref{logdiffchalemma1}) and (\ref{logdiffchalemma2}).
\end{proof}

\begin{prop}\label{logdiffcha}
The following statements are equivalent for any series $f(t)=1+\sum_{i=1}^{\infty}a_i t^i$ and $g(t)=\sum_{i=1}^{\infty}b_i$ in $R[[t]]$.
\begin{enumerate}
\item $g(-t)=-t\dfrac{d}{dt}\log f(t) $.
\item $b_n=Q_n(a_1,\dots,a_n)$ for $n\geq 1$.
\item $\displaystyle \sum_{i=0}^{n-1}(-1)^i a_i b_{n-i}=(-1)^{n+1}na_n$ for $n\geq 1.$
\end{enumerate}
\end{prop}
\begin{proof} Let $f_n$ be $1+a_1t+a_2t^2+\cdots+a_nt^n\in \Lambda(R)$ for each $n\in\mathbb{N}$. The coefficients of $t^n$ in $-t\dfrac{d}{dt}\log f(t)$ is a polynomial with integer coefficients in $a_1,\dots,a_n$, then it is equal to the coefficients of $t^n$ in $f_n(t)$. Then, we can see that the coefficient of $t^n$ in  $-t\dfrac{d}{dt}\log f(t)$ is equal to $(-1)^nQ_n(a_1,\dots,a_n)$ from Lemma \ref{logdiffchalemma}\ $(1)$. Hence we finish to proof of that $(1)$ and $(2)$ are equivalent.

Next, we consider that $(1)$ and $(3)$ are equivalent. From $(1)$, we have
\[ \sum_{i=1}^{\infty}b_i(-t)^i=-t\dfrac{d}{dt}\log\Big(1+\sum_{i=1}^{\infty}a_it^i\Big)=\Big(-\sum_{i=0}^{\infty}ia_it^i\Big)\Big(1+\sum_{i=1}^{\infty}a_it^i\Big)^{-1}\]
Then, we obtain 
\[  \Big(\sum_{i=1}^{\infty}(-1)^ib_it^i\Big)\Big(1+\sum_{i=1}^{\infty}a_it^i\Big)=-\sum_{i=0}^{\infty}ia_it^i \]
Now we compare the coefficients of $t^n$ and multiple by $(-1)^n$, we have the equalities $(3)$ for each $n\in\mathbb{N}$. we can also see the proof of (1) assumption (3) by following this reverse.
\end{proof}

We define maps $z_n$ and $z_t$, which is defined by the logarithmic derivation, and consider one of the condition of commutative ring, $\Ztorsionfree$ and $\Qalg$.
\begin{defi}
A commutative ring $R$ is said to be $\Ztorsionfree$ if whenever  $na=0$ for some element $a\in R$ and integer $n\neq 0$, then $a=0$. and a commutative ring $R$ is said to be $\Qalg$ if it is equipped with ring homomorphism from the rational numbers set $\mathbb{Q}$ to $R$.
\end{defi}
We define maps $f_n:R\rightarrow R$ by $f_n(r)=nr\ (r\in R)$ for each integer $n\neq 0$. Then, the map  $f_n$ is injective for any integer $n\neq 0$ if and only if $R$ is $\Ztorsionfree$.

Moreover, If $R$ is $\Qalg$, we can define $\dfrac{r}{n}:=\dfrac{1}{n}r\in R$  for any $r\in R$ and integer $n\neq0$. In particular, $R$ is $\Ztorsionfree$.

\begin{defi}
For each $n\in\mathbb{N}$, we define maps $z_n:\Lambda(R)\rightarrow R$ by that  $z_n(f)$ is the coefficient of $t^n$ in $-t\dfrac{d}{dt}\log(f(t))$ multiplied $(-1)^n$.
And, we define a map $z_t:\Lambda(R)\rightarrow \Map{\mathbb{N}}{R}$ by $z_t(f)(n):=z_n(f)$ for each strict power series $f\in\Lambda(R)$.
\end{defi}

Since its Definition, we have
\begin{eqnarray}\label{logdiffz}
\sum_{i=1}^{\infty}z_n(f)(-t)^i=-t\dfrac{d}{dt}\log f(t)
\end{eqnarray}
for each $f\in\Lambda(R)$.

 For a ring $R$, $\Map{\mathbb{N}}{R}$, which is used in next Proposition, is a commutative ring with the operation $(\alpha_1+\alpha_2)(n)=\alpha_1(g)+\alpha_2(g),\ (\alpha_1\alpha_2)(g)=\alpha_1(g)\alpha_2(g)$.

\begin{prop}\label{logdiffchaz}
The following statements hold about maps $z_t$ and $z_n$.
\begin{enumerate}
\item We have $z_n(f)=Q_n(a_1,\dots,a_n)$ for any $\displaystyle f=1+\sum_{i=1}^{\infty}a_it^i\in\Lambda(R)$
\item The map $z_t$ is the homomorphism from the multiplication of $\Lambda(R)$ to the addition of $\Map{\mathbb{N}}{R}$. In particular, the map $z_n$ is the homomorphism from the multiplication of $\Lambda(R)$ to the addition of $R$, for each $n\in\mathbb{N}$.
\item If $R$ is $\mathbb{Z}$-torsion-free, then the map $z_t$ is injective.
\item If $R$ is $\mathbb{Q}$-algebra, then the map $z_t$ is bijective.
\end{enumerate}
\end{prop}

\begin{proof} $(1)$ It is clear from Proposition\ \ref{logdiffcha}\ (2) and (\ref{logdiffz}).

$(2)$ Since we have $z_t(fg)=z_t(f)+z_t(g)$ for each $f,g\in\Lambda(R)$ from (\ref{logdiffhom}) and (\ref{logdiffz}), it is proved.

$(3)$ Assume that $R$ is $\Ztorsionfree$. Let $h$ be a strict power series over $R$ with $z_t(h)=0$. we have
\[ 0=-t\dfrac{d}{dt}\log(h)=h'h^{-1} \]
which imply $h'=0$, and we have $h=0$ by assumption that $R$ is $\Ztorsionfree$.

$(4)$ Assume that $R$ is $\Qalg$, and let $\alpha$ be a element of $\Map{\mathbb{N}}{R}$. So, we let $a_n$ be the following equality by induction of $n\in\mathbb{N}$.
\begin{eqnarray*}
a_1:=\alpha(1)\ (n=1),
a_n:=\dfrac{(-1)^{n+1}}{n}\Big(\sum_{i=0}^{n-1}(-1)^i a_i \alpha(n-i)\Big)\ (n\geq 2).   
\end{eqnarray*}
Then, we have
\[ \sum_{i=1}^{\infty}\alpha(i)(-t)^i=-t\dfrac{d}{dt}\log\Big(1+\sum_{i=1}^{\infty}a_it^i\Big) \]
from Proposition\ \ref{logdiffcha}, and we obtain  
\[ z_t\Big(1+\sum_{i=1}^{\infty}a_it^i\Big)=\alpha ,\]
which imply that the function $z_t$ is $\mathrm{surjective}$. Hence, we have that the function $z_t$ is $\mathrm{bijective}$ with (3).
\end{proof}

\newpage

\section{$\lambda$-operation}\label{lambdasection}

In this section, we describe the definition of a \prelambdaring, which has axioms less than a $\lambda$-ring. we can see more in \cite{bu-4} $\S 13$, \cite{bu-1},\ \cite{bu-9},\ and \cite{bu-5}.

\subsection{$\lambda$-operation}
\begin{defi}
A \lambdasemiring\ is a $\semiring \ S$ together with functions $\lambda^n:S\rightarrow S\ (n=0,1,2,\dots)$ called $\lambda$-$\mathrm{operation}$, such that for all $r,s\in S$, the following axioms are satisfied.
\begin{itemize}
\item[($L1)$]  $\lambda^0(r)=1.$
\item[$(L2)$]  $\lambda^1=\mathrm{id}_S. $
\item[$(L3)$]  $\lambda^n(r+s)=\displaystyle\sum_{i+j=n}\lambda^i(r)\lambda^j(s).$
\end{itemize}
We call $S$ a \prelambdaring\ if $S$ is a commutative ring.
\end{defi}

Let $R$ be a commutative ring,  and let $\lambda^n:R\rightarrow R\ (n=0, 1,2,\dots)$ be maps. We define a map $\lambda_t:R\rightarrow R[[t]]$ by
\begin{eqnarray}\label{lambdateq}
\lambda_t(r):=\sum_{i=0}^{\infty} \lambda^i(r)t^i=\lambda^0(r)+\lambda^1(r)t+\lambda^2(r)t^2+\cdots\ \ (r\in R)
\end{eqnarray}
 which considered as a formal power series in $t$ with coefficients in $R$. It is clear that axiom (L1) holds if and only if $\mathrm{Im}(\lambda_t)\subset\Lambda(R)$, and axiom (L3) holds if and only if $\lambda_t(r+s)=\lambda_t(r)\lambda_t(s)$ for all $r,s\in R$. 
 Now, we obtain the next lemma.

\begin{lemm}\label{makelambdaring}
 Let $R$ be a commutative ring,  and let $\lambda^n:R\rightarrow R\ (n=0, 1,2,\dots)$be maps. We define the function $\lambda_t:R\rightarrow R[[t]]$ by $(\ref{lambdateq})$, so the followings statements are equivalent.
\begin{enumerate}
\item A ring $R$ is a \prelambdaring\ together with functions $\{\lambda^n\}_{n=0,1,2,\dots}$
\item $\lambda_t(r+s)=\lambda_t(r)(s)$, and $\lambda_t(r)=1+rt+\cdots$ for all $r,s\in R$.

\end{enumerate}
\end{lemm}

\begin{exa}\label{Examplebinomial}
The rings $\mathbb{Z}$ of integers is a \prelambdaring with $\lambda_t(n)=(1+t)^n$ for $n\in\mathbb{Z}$. we will consider in detail with \S\ref{lambdabinomial}.
\end{exa}

\begin{exa}\label{representationspace}
For a finite group $G$, let $M_{\mathbb{C}}(G)$ denote the set of isomorphism classes $\{V\}$ of finite dimension representation of $G$ over $\mathbb{C}$. The operation  $\{V\}+\{W\}:=\{V\oplus W\}$ and $\{V\}\{W\}:=\{V\otimes_{\mathbb{C}} W\}$ make $M_{\mathbb{C}}(G)$ into a semiring. the functions $\lambda^n: M_{\mathbb{C}}(G) \rightarrow M_{\mathbb{C}}$ ,$\lambda^n(\{V\})=\{\bigwedge^n(V)\}$ defines a \lambdasemiring structure on $M_{\mathbb{C}}(G)$. we will consider in detail with \S\ref{charactersection}.
\end{exa}

\subsection{Adams operation}
 In this section, we define Adams operations in a \prelambdaring and establish some of their properties. Moreover, we describe a pre-$\psi$-ring, which is a ring $R$ together with ring endomorphisms of addition $\psi^k$ that behave like the Adams operations in a \prelambdaring, and it gives a \prelambdaring structure on $\Qalg$.

\begin{defi}\label{Adams0}
 For each $n=1,2,3,\dots$ define the n-th Adams operation $\psi^n:R\rightarrow R$ by  $\psi^n=z_n\circ\lambda_t$. In other words, we defined the Adams operation by the following equality holds
\[ \psi_{-t}(r)=-t\dfrac{d}{dt}\log(\lambda_t(r)) \label{Adams}\]
for $r\in R$, where $\psi_t (r)=\sum_{i=1}^{\infty}\psi^k(r)t^k.$
\end{defi}

\begin{prop} \label{Adams1}
 The following statements hold in any \prelambdaring\ $R$. 
\begin{enumerate}
\item $\psi^n(r)=Q_n(\lambda^1(r),\dots,\lambda^n(r)).$ \label{AdamsQ}
\item $\psi^n(r)-\lambda^1(r)\psi^{n-1}(r)+\cdots+\lambda^{n-1}(r)\psi^1(r)=(-1)^{n+1}n\lambda^n(r).$\label{Newtonsformula}
\end{enumerate}
for all $n=1,2,3,\dots, r\in R$.
\end{prop}
\begin{proof}
Statement (1) follows from the definition and Proposition \ref{logdiffchaz}\ $(1)$,\ and statement (2) follows from the definition and Proposition \ref{logdiffcha}.

 Statement (2) of Proposition\ \ref{Adams1} is called Newtons formula since symmetric polynomial ring $\Lambda(X^n)$ is a $\lambda$-ring for all $n=1,2,3,\dots$, and (2) gives a Newtons formula in symmetric polynomial.
\end{proof}

\begin{prop}\label{Adamshomo}
Adams operation $\psi^1$ is the identity map of $R$. and each Adams operation is a homomorphism of the addition.  
\end{prop}
\begin{proof} 
 We put $n=1$ in statement (1) of Proposition \ref{Adams1}. Since Lemma\ \ref{makelambdaring} and\ Proposition \ref{logdiffchaz}(2), $\psi^n$ is additive.  
\end{proof}

Next, we define a pre-$\psi$-ring, and we describe it's properties and $\Qalg$'s properties , which product a structure of \prelambdaring.

\begin{defi}
 A $\mathrm{pre}$-$\psi$-$\mathrm{ring}$ is a commutative ring $R$ together with homomorphism of the addition $\psi^n:R\rightarrow R\ (n=1,2,\dots)$ such that $\psi^1$ is the identity map of $R$.
\end{defi}

\begin{prop}\label{Adamsmakelambda}
 Let $S$ be a pre-$\psi$-ring and a $\Qalg$. Then, there exists a structure of \prelambdaring on $S$ such that Adams operation of this structure is $\psi^n$. Moreover, this structure is required to be unique.
\end{prop}
\begin{proof} 
 We define a map $\Psi:S\rightarrow \Map{\mathbb{N}}{S}$ by setting $\Psi(r)(n)=\psi^n(r)$ for all $r\in S$ and $n\in\mathbb{N}$. It is a homomorphism from the addition of $S$ to the addition of $\Map{\mathbb{N}}{S}$. Furthermore, the map $z_t$ is a isomorphism from the multiplication of $\Lambda(S)$ to the addition of $\Map{\mathbb{N}}{S}$ from Proposition \ref{logdiffchaz}. Hence, we let $\lambda_t$ be $\lambda_t:=z_t^{-1}\circ\Psi$. it is a homomorphism from the addition of $S$ to from the multiplication of $\Lambda(S)$, and
\[ \lambda_t(r)=z_t^{-1}\circ\Psi(r)=1+rt+\cdots \]
for any $r\in R$. so, $\lambda_t$ gives a structure of \prelambdaring to $R$, it's n-th Adams operation is $\psi^n$ from Lemma \ref{makelambdaring}. Uniqueness follows that a map $\dfrac{d}{dt}\log$ is bijective.
\end{proof}

\begin{exa}\label{MapGC}
For a finite group $G$, $\MapC{G}$ is a $\mathbb{C}$-$\mathrm{algebra}$ with the operation $(f_1+f_2)(g)=f_1(g)+f_2(g),\ (f_1f_2)(g)=f_1(g)f_2(g),\ (rf_1)(g)=rf_1(g)\ (g\in G)$. In paticular, it is a $\Qalg.$ We define the functions $\psi^n:\MapC{G}\rightarrow\MapC{G}$ by
\begin{eqnarray}\label{MapGCdef}
\psi^n(f)(g):=f(g^n)\ (f\in\MapC{G},\ g\in G )
\end{eqnarray}
for each $n$=1,2,3,\dots . Hence, $\MapC{G}$ becomes to a pre-$\psi$-ring, and be given a structure of \prelambdaring such that Adams operation is (\ref{MapGCdef}).
\end{exa}

\subsection{Symmetric powers operation}
 In this section, we describe symmetric powers operation in a \prelambdaring\ 
$R$, which is discussed in \cite{bu-5}. it will be used in symmetric powers representation.

\begin{defi}\label{defofsymm}
 We define the function\ $S_t:R\rightarrow R[[t]]$  by setting
\[ S_t(r):=(\lambda_{-t}(r))^{-1}=\Big(1+\sum_{i=1}^{\infty}\lambda^i(r)(-t)^i\Big)^{-1}\ (r\in R).\]
and we define symmetric powers operation $S^n: R\rightarrow R$ by that $S^n(r)$ is coefficient of  $t^n$ in $S_t(r)$ for each $n\geq 0$.
\end{defi}
It is clear that $S^0$ is a constant map of unit element of $R$.

\begin{prop}\label{formulaanother}
The following statements hold in any \prelambdaring\ $R$ for any $n\geq 0$ .
\[ S^n+(-1)\lambda^1 S^{n-1}+\cdots+(-1)^{n-1}\lambda^{n-1}S^1+\lambda^n=0 \]
\end{prop}
\begin{proof}
It follows from definition $S_t\lambda_{-t}=1$.
\end{proof}

\begin{prop}\label{Hyperlambda}
For any $r\in R$ and $n\in\mathbb{N}$, $S^n(r)$=$h_n(\lambda^1(r),\dots,\lambda^n(r))$ holds. In paticular, $S^1$ is the identity map of $R$.
\end{prop}
\begin{proof} By Proposition \ref{h-symm}, we have
\begin{eqnarray*}
S_t(r)=(\lambda_{-t}(r))^{-1}&=&1+\sum_{n=1}(-1)^nh_n(\lambda^1(r),\dots,\lambda^n(r))(-t)^n \\
&=&1+\sum_{n=1}h_n(\lambda^1(r),\dots,\lambda^n(r))t^n.
\end{eqnarray*}
for any $r\in R$. Thus, comparing the coefficients of $t^n$, we obtain the first statement.
\end{proof}

\begin{prop}
In any \prelambdaring\ $R$, we have $S_t(r+s)=S_t(r)S_t(s)$, and $S^n(r+s)=\displaystyle\sum_{i+j=n}S^i(r)S^j(s)$ for $r,s\in R$, $n\in\mathbb{N}$.
\end{prop}
\begin{proof}First statement follows from $S_t(r+s)=(\lambda_{-t}(r+s))^{-1}$=$(\lambda_{-t}(r))^{-1}(\lambda_{-t}(s))^{-1}=S_t(r)S_t(s)$. By comparing of coefficient of $t^n$ in it, we have a second statement.
\end{proof}

\subsection{$\lambda$-operation and homomorphism}
We discuss the morphism of  \lambdasemiring and  \prelambdaring, and describe the relation of ring complement of semiring.
\begin{defi}
\begin{enumerate}
\item Let $S$ and $S'$ be semirings. A semi-$\lambda$-homomorphism\ $f:S\rightarrow S'$ is a morphism of semiring such that $f\circ\lambda^n=\lambda^n\circ f$ for $n\geq 0.$
\item Let $R$ and $R'$ be \prelambdaring s. A pre-$\lambda$-homomorphism\ $f:R\rightarrow R'$ is a ring homomorphism such that $f\circ\lambda^n=\lambda^n\circ f$ for $n\geq 0.$
\end{enumerate}
\end{defi}

About such morphism, we have next statements.

\begin{prop}\label{lambdafcha1}
\begin{enumerate}
\item The identity map of \lambdasemiring $S$ is a semi-$\lambda$-homomorphism.
\item Let $S$, $S'$ and $S''$ be \lambdasemiring s, and $f:S\rightarrow S'$ and  $g:S'\rightarrow S''$ be semi-$\lambda$-homomorphisms. then $g\circ f:S\rightarrow S''$ is a semi-$\lambda$-homomorphism.
\end{enumerate}
In addition , These statements holds about \prelambdaring and pre-$\lambda$-homomorphism.
\end{prop}

Thereafter, we call $\mathrm{semi}$-$\lambda$-$\mathrm{homomorphism}$ and pre-$\lambda$-homomorphism to \lambdaf{homomorphism}. And, If \lambdaf{homomorphism} is injective, surjective, and bijective, we call it to \lambdaf{monomorphism},\ \lambdaf{epimorphism},\ \lambdaf{isomorphism} each other.

\begin{prop}\label{lambdafcha2}
Let $R$ and $R'$ be a \prelambdaring s, and let $f:R\rightarrow R$ be a ring homomorphism.
\begin{enumerate}
\item If $f:R\rightarrow R'$ is a \lambdaf{homomorphism},\ we have $\psi^n\circ f=f\circ\psi^n$ , and $S^n\circ f=f\circ S^n$ for any $n\in\mathbb{N}$.
\item Conversely, If $f$ is satisfied $\psi^n\circ f=f\circ\psi^n$ for any $n\in\mathbb{N}$ and $R'$ is a $\mathbb{Z}$-torsion-free, $f$ is \lambdaf{homomorphism}.
\end{enumerate}
\end{prop}
\begin{proof} $\ (1)$ From (\ref{AdamsQ}) of Proposition \ref{Adams1}, we have\ 
\begin{eqnarray*}
\psi^n(f(r))&=&Q_n(\lambda^1(f(r)),\dots,\lambda^n(f(r)))\\
&=&Q_n(f(\lambda^1(r)),\dots,f(\lambda^n(r)))\\
&=&f(Q_n(\lambda^1(r),\dots,\lambda^n(r)))=f(\psi^n(r))
\end{eqnarray*}
for any $r\in R$, $n\in\mathbb{N}$, Similarly, we obtain
\begin{eqnarray*}
S^n(f(r))&=&h_n(\lambda^1(f(r)),\dots,\lambda^n(f(r)))\\
&=&h_n(f(\lambda^1(r)),\dots,f(\lambda^n(r)))\\
&=&f(h_n(\lambda^1(r),\dots,\lambda^n(r)))=f(S^n(r))
\end{eqnarray*}
from $\mathrm{Proposition\ \ref{Hyperlambda}}$.

$(2)$ We prove by induction of $n\in\mathbb{N}$. $\lambda^1$ is the identity map of $R$, so it is clear about $n=1$. Suppose by induction that $f\circ\lambda^j=\lambda^j\circ f\ (j=1,\dots,n-1).$ By the induction hypothesis and Newtons formula , we have
\begin{eqnarray*}
(-1)^{n+1}n\lambda^n(f(r))&=&\psi^n(f(r))-\lambda^1(f(r))\psi^{n-1}(f(r))+\cdots+\lambda^{n-1}(f(r))\psi^1(f(r)) \\
&=&f(\psi^n(r)-\lambda^1(r)\psi^{n-1}(r)+\cdots+\lambda^{n-1}(r)\psi^1(r))\\
&=&f((-1)^{n+1}n\lambda^n(r))=(-1)^{n+1}nf(\lambda^n(r)).
\end{eqnarray*}
Since $R'$ is $\mathbb{Z}$-torsion-free by assumption, it follows that $\lambda^n(f(r))=f(\lambda^n(r))$. This finishes the induction and the proof.
\end{proof}

\begin{prop}\label{makelambdaring2}
Let $S$ be a \lambdasemiring, and commutative about its multiplication. and let $(S^*, \theta)$ be the ring completion of $S$. then, there exists a structure of \prelambdaring in $S$ such that $\theta$ is a \lambdaf{homomorphism}. Moreover, this structure is required to be unique.
\end{prop}
\begin{proof} 
Since $S$ is commutative about it's multiplication, $S^*$ is too. so, we can consider the formal power series over $S^*$. We define a map $G:S\rightarrow S^*[[t]]$ by setting
\[ G(r):=\sum_{n=0}^{\infty}\theta\circ\lambda^i(r)t^n\ (r\in S). \]
A map $\lambda^0$ in $S$ is the constant map of unit element, so we have $\mathrm{Im}(G)\subset\Lambda(S^*)$. Hence we have
\begin{eqnarray*}
G(r+s)&=&\sum_{n=0}^{\infty}\theta\circ\lambda^n(r+s)t^n \\
&=&\sum_{n=0}^{\infty}\theta\Big(\sum_{i+j=n}\lambda^i(r)\lambda^j(s)\Big)t^n \\
&=&\sum_{n=0}^{\infty}\Big(\sum_{i+j=n}\theta\circ\lambda^i(r)\theta\circ\lambda^j(s)\Big)t^n=G(r)G(s)
\end{eqnarray*}
for all $r,s\in S$, which implies that $G$ is a homomorphism from the addition of $S$ to the multiplication of $\Lambda(S^*)$. Hence, There exists a map $\lambda_t$ uniquely such that it is a homomorphism from the addition of $S^*$ to the multiplication of $\Lambda(S^*)$, and it is satisfied $G=\lambda_t\circ\theta$, from Proposition \ref{semiuniv}(2). In addition, we have
\begin{eqnarray*}
\lambda_t(\inn{a}{b}{})=G(a)G(b)^{-1}&=&(\theta\circ\lambda^0(a)+\theta\circ\lambda^1(a)t+\cdots)(\theta\circ\lambda^0(b)-\theta\circ\lambda^1(b)t+\cdots)\\
&=&(1+\inn{a}{0}{}t+\cdots)(1-\inn{b}{0}{}t+\cdots)=1+\inn{a}{b}{}t+\cdots
\end{eqnarray*}
for all $\inn{a}{b}{}\in S^*$. Therefore we obtain a \prelambdaring structure in $S^*$ with $\lambda_t$.
At last, we have
\begin{eqnarray*}
1+\sum_{i=1}^{\infty}\theta\circ\lambda^i(r)t^i&=&G(r)=\lambda_t\circ\theta(r)=1+\sum_{i=1}^{\infty}\lambda^i\circ\theta(r)t^i
\end{eqnarray*}
for all $r\in R$, hence comparing of coefficient of $t^n$ gives that $\theta:S\rightarrow S^*$ is \lambdaf{homomorphsim}.
\end{proof}

\subsection{Binomial coefficient}\label{lambdabinomial}
Let $R$ be a $\Qalg$ in this section, we prepare of a strict power series $(1+t)^r\in\Lambda(R)$ where $r\in R$, and describe its inverse of multiplication and derivative. We will be used it in \S\ref{calculateresult1}.
\begin{defi}\label{binomial}
We define a binomial symbols
\begin{eqnarray}\label{binomial1}
\binom{r}{n}:=\dfrac{r(r-1)\cdots(r-(n-1))}{n!}
\end{eqnarray}
actually lie in $R$ for $r\in R$ and $n\in\mathbb{N}$. and we define 
\begin{eqnarray}\label{binomial2}
(1+t)^r:=\sum_{i=1}^{\infty}\binom{r}{i}t^i.
\end{eqnarray}
\end{defi}

Next, we describe inverse of multiplication of $(1+t)^r$ is $(1+t)^{-r}$.

\begin{lemm}\label{binomhomolemma}
The following statements hold for all $r\in R,\ n\in\mathbb{N}$.
\[ (-1)^{n+1}r\ \Big(\sum_{i=0}^{n-1}(-1)^i\binom{r}{i}\ \Big)=n\binom{r}{n}.\]
\end{lemm}
\begin{proof}
 We prove by induction of $n\in\mathbb{N}$. If $n=1$, it is clear. Suppose by induction that is a statement where $n-1$. we have
\begin{eqnarray*}
(-1)^{n+1}r\ \Big(\sum_{i=0}^{n-1}(-1)^i\binom{r}{i}\ \Big)&=&-(n-1)\binom{r}{n-1}+r\binom{r}{n-1}=n\binom{r}{n}
\end{eqnarray*}
for all $r\in R$, This finishes the induction and the proof.
\end{proof}

\begin{prop}\label{binomialhomo}
The equalities $(1+t)^r(1+t)^s=(1+t)^{r+s}$ holds for any $r,s\in R$. In particular,  $((1+t)^r)^{-1}=(1+t)^{-r}$ holds.
\end{prop}
\begin{proof}
We define $\psi^n:R\rightarrow R$ by $\psi^n:=\mathrm{id}_R$ for each $n\in\mathbb{N}$. then, $R$ becomes to $\mathrm{pre}$-$\psi$-$\mathrm{ring}$ together with the series $\{ \psi^n \}_{n=1,2,3,\dots}$. Since Proposition\ \ref{Adamsmakelambda}, $R$ is given a \prelambdaring structure uniquely such that $\psi^n$ is n-th Adams operation. we show that it's $\lambda$-operation is satisfied
\begin{eqnarray}\label{binomlambda}
\lambda^n(r)=\binom{r}{n} 
\end{eqnarray}
for all $r\in R$, $n=0,1,2,\dots$ by induction of $n$. When $n=0,1$, it is clear since  properties of $\lambda$-$\mathrm{operation}$. Suppose by induction that $\lambda^j(r)=\binom{r}{j}\ (j=0,1,2,\dots, n-1).$ By using  Newtons formura(\ref{Newtonsformula}) and Lemma\ \ref{binomhomolemma}, we compute
\begin{eqnarray*}
(-1)^{n+1}n\lambda^n(r)&=&\psi^n(r)-\lambda^1(r)\psi^{n-1}(r)+\cdots+(-1)^{n-1}\lambda^{n-1}(r)\psi^{1}(r)\\
&=&r\Big(\sum_{i=0}^{n-1}(-1)^{i}\binom{r}{i}\Big)=(-1)^{n+1}n\binom{r}{n}
\end{eqnarray*}
and divide by $(-1)^{n+1}n$, we obtain (\ref{binomlambda}) for all $r\in R$. This finishes the induction. In particular, from property of $\lambda_t$, we have 
\[ (1+t)^{r+s}=\lambda_t(r+s)=\lambda_t(r)\lambda_t(s)=(1+t)^r(1+t)^s\]
for $r,s\in R$.
\end{proof}

Next, we describe the logarithmic derivative of $(1+t)^r$.

\begin{lemm}\label{binomialdiff}
The following equality holds for all $a,r\in R$, $n\in\mathbb{N}$.
\[-t\dfrac{d}{dt}\log(1-a(-t)^n)^{\frac{r}{n}}=\sum_{i=1}^{\infty}ra^i(-t)^{in}\]
\end{lemm}
\begin{proof}
First, derivation of $(1+t)^r$ gives
\begin{eqnarray}\label{binomialdiffproof}
\dfrac{d}{dt}(1+t)^r=r(1+t)^{r-1}
\end{eqnarray}
for $r\in R$. Since $\mathrm{(\ref{binomialdiffproof})}$,\ $\mathrm{Proposition\ \ref{diff}\ (3)}$と$\mathrm{Proposition\ \ref{binomialhomo}}$, we have
\begin{eqnarray*}
-t\dfrac{d}{dt}(1-a(-t)^n)^{\frac{r}{n}}&=&(-t)\dfrac{r}{n}(1-a(-t)^n)^{\frac{r}{n}-1}na(-t)^{n-1}\ (1-a(-t)^n)^{-\frac{r}{n}} \\
&=&ra(-t)^n(1-a(-t)^n)^{-1}=\sum_{i=1}^{\infty}ra_i(-t)^{in}
\end{eqnarray*}
for $r\in R$, $n\in\mathbb{N}$.
\end{proof}

Last, we describe the coefficients of $t^n$ in $(1+t)^{-r}$ for any $r\in R$. In general, let $f$ be a power series f over $R$, we have
\[ f(t)=\sum_{i=0}^{\infty}\dfrac{f^{(n)}(0)}{n!}t^n \]
which $f^{(n)}(0)$ is a constant term of f that is differentiated $n$ times. we put $f=(1+t)^{-r}$, then we have
\[ \dfrac{f^{(n)}(0)}{n!}=\dfrac{(-r)(-r-1)\cdots(-r-(n-1))}{n!}=(-1)^n\binom{r+(n-1)}{n} \]
which imply that
\[ (1+t)^{-r}=\sum_{n=0}^{\infty}\binom{r+(n-1)}{n}(-t)^n .\]
In particular, we put $t$ to $-t$ in this statement, we get
\begin{eqnarray}\label{inverseformula}
 (1-t)^{-r}=\sum_{n=0}^{\infty}\binom{r+(n-1)}{n}t^n
\end{eqnarray}
for $r\in R$.
\newpage

\section{Corresponding of  character}\label{charactersection}

In this chapter, we describe how to calculate a multiplicity of irreducible characters of the symmetric and exterior powers representation.

At first, we consider the one of Lemma. In general, let $V$ be a vector space over $\mathbb{C}$ which has a finite dimension $m=\dim V$, let $f:V\rightarrow V$ be a linear map, and we define $\mathrm{Tr}(f)$ by the trace of the linear map $f$. About this values, Tr(f) is equal to all eigenvalue of $f$, which is $\alpha_1,\dots,\alpha_n$.

Next, we define maps $f^n, S^n(f), \Lambda^n(f)$ by the following equalities hold for each $n\in\mathbb{N}$.
\begin{eqnarray}
f^n&:&V\rightarrow V,\ f^n:=\overbrace{f\circ\cdots\circ f}^{n} \\
\label{symmf}S^n(f)&:&S^n(V)\rightarrow S^n(V),\\ \nonumber&&S^n(f)(v_1\cdots v_n):=f(v_1)\cdots f(v_n),\\ 
\label{lambdaf}\Lambda^n(f)&:&\textstyle \bigwedge^n(V)\rightarrow\ \textstyle \bigwedge^n(V),\\ \nonumber &&\Lambda^n(f)(v_1\wedge\cdots\wedge v_n):=f(v_1)\wedge\cdots\wedge f(v_n).
\end{eqnarray}

The vector spaces $V$, $S^n(V)$ and $\Lambda^n(V)$ have a finite dimension. About a  trace of these linear map, we obtain the following equalities.
\begin{eqnarray}
\label{powerstr} \mathrm{Tr}(f^n)&=&\alpha_1^n+\cdots+\alpha_m^n=Q_n(\beta_1,\dots,\beta_m),\\
\label{symmtr}\mathrm{Tr}(S^n(f))&=&\sum_{1\leq i_1\leq \cdots \leq i_n\leq m}\alpha_{i_1}\cdots\alpha_{i_n}=h_n(\beta_1,\dots,\beta_m), \\
\label{lambdatr}\mathrm{Tr}(\Lambda^n(f))&=&\sum_{1\leq i_1<\cdots <i_n\leq m}\alpha_{i_1}\cdots\alpha_{i_n}=\beta_n
\end{eqnarray}
where $\beta_i$ be a elementary symmetric polynomial on $\alpha_1,\dots,\alpha_m$ for each $i \in\mathbb{N}$.

In this chapter afterward, we let $G$ be an arbitrary finite multiplicative group with identity element, and we call the representation that has a finite dimension over $\mathbb{C}$. If we let $\rho:G\rightarrow GL(V)$ be the representation of $G$ for each $i\geq 0$, we can give the $i$-th symmetric power representation $S^i(V)$ and the $i$-th exterior representation $\bigwedge^i(V)$ that is defined by
\begin{eqnarray}
\label{symmrep}S^i\rho:G\rightarrow GL(S^i(V)),\ (S^i\rho)(g):=S^i(\rho(g)),\\
\label{lambdarep}\Lambda^i\rho:G\rightarrow GL(\textstyle\bigwedge^i(V)),\ (\Lambda^i\rho)(g):=\Lambda^i(\rho(g)).
\end{eqnarray}
from (\ref{symmf}),\ (\ref{lambdaf}), where $g\in G$. Remark that if the representation $V$ is isomorphic to the representation $W$ as the representation of $G$, then $S^i(V)$ is isomorphic to $S^i(W)$, and $\bigwedge^i(V)$ is isomorphic to $\bigwedge^j(W)$, as the representation of $G$, too. 

 As described under Example \ref{representationspace},\ the set of the all representation of $G$ has a equivalence relation that is defined by the isomorphism as the representaiton of $G$. Let $M_{\mathbb{C}}(G)$ be the set of isomorphism classes $\{V\}$ of finite dimension representation of $G$ over $\mathbb{C}$. The operation  $\{V\}+\{W\}:=\{V\oplus W\}$ and $\{V\}\{W\}:=\{V\otimes_{\mathbb{C}} W\}$ make $M_{\mathbb{C}}(G)$ into a semiring. the functions $\lambda^n: M_{\mathbb{C}}(G) \rightarrow M_{\mathbb{C}}$, $\lambda^n(\{V\})=\{\bigwedge^n(V)\}$ defines a \lambdasemiring structure on $M_{\mathbb{C}}(G)$. 

Now, let state once again the purpose of this paper here. Let $V_1,V_2,\dots,V_k$ be all of representation spaces of the irreducible representation of $G$ over $\mathbb{C}$.   Since the representation $S^i(V)$ and $\bigwedge^i(V)$ are the representation of $G$ and over $\mathbb{C}$, these representations are completely reducible from Maschkes Theorem, and isomorphic to direct sum of some $V_1,V_2,\dots,V_k$ as the representation of $G$. In addition, representations $S^i(V)$ and $\bigwedge^i(V)$ have a finite dimension over $\mathbb{C}$ too since $V$ has a finite dimension over $\mathbb{C}$. Hence, these representation are isomorphic to direct sum of irreducible representation, and every irreducible component appears finitely, and uniquely. On the other words, we have
\begin{eqnarray}\label{irrdecomp1}
\begin{split}
S^i(V)\cong \overbrace{V_1 \oplus \cdots \oplus V_1}^{m_{i,1}}\oplus \overbrace{V_2 \oplus \cdots \oplus V_2}^{m_{i,2}}\oplus\cdots\oplus\overbrace{V_k \oplus \cdots \oplus V_k}^{m_{i,k}} \\
\textstyle\bigwedge^i(V)\cong \overbrace{V_1 \oplus \cdots \oplus V_1}^{n_{i,1}}\oplus\overbrace{V_2 \oplus \cdots \oplus V_2}^{n_{i,2}}\oplus\cdots\oplus\overbrace{V_k \oplus \cdots \oplus V_k}^{n_{i,k}}
\end{split}
\end{eqnarray}
and the nonnegative numbers $m_{i,1},\dots,m_{i,k},n_{i,1},\dots,n_{i,k}$ are determined finitely, and uniquely.  The purpose of this paper is that calculating about it.

 To achieve this purpose, we consider the character of the representation $\rho:G\rightarrow GL(V)$, which is a element of $\MapC{G}$, and is reffed that maps from $g\in G$ to the trace of $\rho(g)$. If the representation $V$ is isomorphic to the representation $W$, the character of $V$ is equal to the character of $W$. Hence, we can define a map $X:M_{\mathbb{C}}(G)\rightarrow \MapC{G} $ which maps from the representation $V$ to the character of it. A map $X$ is injective from char($\mathbb{C}$)=$0$ (from \cite{bu-3}, Corollary 2.3.7). 

In addition, as described under Example \ref{MapGC}, $\MapC{G}$ is a $\mathbb{C}$-$\mathrm{algebra}$ with the operation $(f_1+f_2)(g)=f_1(g)+f_2(g),\ (f_1f_2)(g)=f_1(g)f_2(g),\ (rf_1)(g)=rf_1(g)\ (g\in G)$. In particular, it is a $\Qalg.$ A map $X:M_{\mathbb{C}}(G)\rightarrow\MapC{G}$ is a ring\ monomorphism with these operator. Let\ $X_1,\dots,X_k$ be characters of $V_1,\dots,V_k$, on the other hand, let $\chi_j$ be X($\{V_j\}$) for each $j=1,\dots,k$. In this case, the equalities (\ref{irrdecomp1}) is equivalent to
\begin{eqnarray}\label{irrdecomp2}
\begin{split}
X(\{S^i(V)\})&=m_1\chi_1+\cdots+m_k\chi_k,\\
X(\{\textstyle\bigwedge^i(V)\})&=n_1\chi_1+\cdots+n_k\chi_k.
\end{split}
\end{eqnarray}

At last, we consider about $X(\{S^i(V)\})$ and $X(\{\textstyle\bigwedge^i(V)\})$, which are characters of the symmetric powers representation and the exterior powers representation. About these, we have the following Proposition. \\

The following equalities are hold for any $i=0,1,2,\dots, $ and for any representation of $G$.
\begin{eqnarray}\label{irrdecomp3}
\begin{split}
(1&)&\ X\big(\{\textstyle\bigwedge^i(V)\}\big)=\lambda^i\big(X(\{ V\})\big),\\
(2&)&\ X\big(\{S^i(V)\}\big)=S^i\big(X(\{V\})\big).
\end{split}
\end{eqnarray}
We describe the proof of this Proposition. About $(1)$, See more in \cite{bu-1}.

\begin{proof} Let $(R_{\mathbb{C}}(G), \theta)$ be the ring completion of  $M_{\mathbb{C}}(G)$. In this case, $R_{\mathbb{C}}(G)$ has a structure of  a \prelambdaring uniquely such that a map $\theta:M_{\mathbb{C}}(G)\rightarrow :R_{\mathbb{C}}(G)$ is a \lambdaf{homomorphism} from Proposition \ref{makelambdaring}.

The claim of $(1)$ is that a map $X:M_{\mathbb{C}}(G)\rightarrow\MapC{G}$ is a \lambdaf{homomorphism}. At first, there is a $\mathrm{ring homomorphism}$ $X':R_{\mathbb{C}}(G)\rightarrow\MapC{G}$ uniquely such that the following equalities holds
\begin{eqnarray}\label{repchi}
X=X'\circ\theta
\end{eqnarray}
from Proposition \ref{semiuniv}. A \lambdaf{homomorphism}\ $\theta$ can exchange the  $\lambda$-operation,\ so If we have  $X'\circ\lambda^i(\theta(\{V\}))=\lambda^i\circ X'(\theta(\{V\}))$ for any $\{V\}\in M_{\mathbb{C}}(G)$, then we can finish the proof of $(1)$ from Proposition \ref{lambdafcha1}. Since $\MapC{G}$ is the $\Ztorsionfree$, it suffered to show that
\begin{eqnarray}\label{irrdecompfinal}
X'\circ\psi^i(\theta(\{V\}))=\psi^i\circ X'(\theta(\{V\})) 
\end{eqnarray}
for any $i\in\mathbb{N}$, it is satisfied from Proposition\ \ref{lambdafcha2}.
 
Now, we prove (\ref{irrdecompfinal}). Let $\rho:G\rightarrow GL(V)$ be the representation of $G$ with $m:=\dim V$, let $\alpha_1,\dots,\alpha_m$ be eigenvalues of $\rho(g)$ for each $g\in G$, and let $\beta_i$ be a $i$-th elementary symmetric polynomial on $\alpha_1,\dots,\alpha_m$. Since\ Proposition (\ref{Adams1}) (1),\  (\ref{repchi}) and (\ref{lambdatr}), we have
\begin{eqnarray}\label{irrdecompproof1}
\begin{split}
X'\circ\psi^i\big(\theta(\{V\})\big)(g)&=X'\Big(Q_i\big(\lambda^1\circ\theta(\{V\}),\dots,\lambda^i\circ\theta(\{V\})\big)\Big)(g) \\
&=Q_i\big(X'\circ\theta\circ\lambda^1(\{V\})(g),\dots,X'\circ\theta\circ\lambda^i(\{V\})(g)\big) \\
&=Q_i\big(X\circ\lambda^1(\{V\})(g),\dots,X\circ\lambda^i(\{V\})(g)\big) =Q_i(\beta_1,\dots,\beta_i)
\end{split}
\end{eqnarray}
and since (\ref{MapGCdef}), (\ref{repchi}), (\ref{powerstr}) , we have
\begin{eqnarray}\label{irrdecompproof2}
\begin{split}
\psi^i\circ X'\big(\theta(\{V\})\big)(g) &= X'\circ\theta(\{V\})(g^i)\\
&=X(\{V\})(g^i)\\
&=\alpha_1^i+\cdots+\alpha_m^i=Q_i(\beta_1,\cdots,\beta_m).
\end{split}
\end{eqnarray}
To prove that (\ref{irrdecompproof1}) equals (\ref{irrdecompproof2}), we consider separately in the case of $i\leq m$ and $i>m$.
\begin{itemize}
\item If $i\leq m$, then we have $Q_i(\beta_1,\dots,\beta_i)=Q_i(\beta_1,\dots,\beta_m)$ from Lemma \ref{logdiffchalemma}$(2)$. 
\item Assume that $i>m$. If $k\in\mathbb{N}$ is satisfied $k\geq m$, then $\beta_k=0$. hence, we have\ $Q_i(\beta_1,\dots,\beta_i)=Q_i(\beta_1,\dots,\beta_m,0,\dots,0)=Q_i(\beta_1,\dots,\beta_m)$ from Lemma\ \ref{logdiffchalemma}$(3)$.
\end{itemize}
Therefore we obtain (\ref{irrdecompfinal}), so we finish to prove $(1)$.

 Next, we consider about $(2)$ with the same notation. Since the result of $(1)$, Proposition\ref{Hyperlambda} and (\ref{symmtr}), we have
\begin{eqnarray*}
X(\{S^i(V)\})(g)&=&h_i(\beta_1,\dots,\beta_i) \\
&=&h_i\big(X(\{ \Lambda^1(V)\})(g),\dots,X(\{\Lambda^i(V)\})(g)\big) \\
&=&h_i\big((X\circ\lambda^1(\{V\}))(g),\dots,(X\circ\lambda^i(\{V\}))(g)\big)\\
&=&h_i\big(\lambda^1\circ X(\{V\}),\dots,\lambda^i\circ X(\{V\})\big)(g)=S^i\big(X(\{ V\})\big)(g)
\end{eqnarray*}
for any $g\in G$. so, we finish to prove $(2)$ too.
\end{proof}

Since (\ref{irrdecomp3}),\ The equalities (\ref{irrdecomp2}) is equivalence to
\begin{eqnarray}\label{irrdecomp4}
\begin{split}
S^i\big(X(\{ V\})\big)&=X(\{S^i(V)\})=m_1\chi_1+\cdots+m_k\chi_k,\\
\lambda^i\big(X(\{ V \})\big)&=X(\{\textstyle\bigwedge^i(V)\})=n_1\chi_1+\cdots+n_k\chi_k.
\end{split}
\end{eqnarray}

 Hence, we obtain the following summary.

Let $V$ be the representation of $G$. To calculate each multiplicity of irreducible representation of the symmetric powers representation $S^i(V)$ and the exterior powers representation $\bigwedge^i(V)$, we calculate $S^i\big(X(\{ V\})\big)$ and $\lambda^i\big(X(\{V\})\big)$, and express these character to sum of irreducible characters.
\newpage

\section{Calculating method}\label{calculateresult1}
Now, we calculate the multiplicity of irreducible component of the symmetric and exterior powers representation of finite groups by some methods. Let $G$ is a finite group, and let $\chi_1,\dots,\chi_k$ be all of irreducible character of $G$. We denote $\chi$ by the character of the representation of $G$, then we calculate nonnegative integers $m_{i,j}$ and $n_{i,j}$ such that the following equalities hold.
\begin{eqnarray}\label{approach}
\begin{split}
S^i(\chi)&=m_{i,1}\chi_1+\cdots+m_{i,k}\chi_k,\\
\lambda^i(\chi)&=n_{i,1}\chi_1+\cdots+n_{i,k}\chi_k
\end{split}
\end{eqnarray}
In addition, we define a character of the trivial representation which has 0-dimension by a zero element of a ring $\MapC{G}$.

By the property of  $\lambda$-operation and symmetric powers operation, the map $S^0$ and  $\lambda^0$ are constant maps to a unit element of $\MapC{G}$, and the map $S^1$ and $\lambda^1$ are identity map of $\MapC{G}$. Therefore, it is clear that the case of  $i=0,1$. And, the case of $i>\chi(1)$, which is the dimension over  of the representation space, $\lambda^i(\chi)$ is $0$ from that dimension over $\mathbb{C}$ of $i$-th exterior power representation is $0$.
 
 And, we define the inner product of $\MapC{G}$ by the following equality.
\begin{eqnarray}\label{innerproduct}
\inn{f}{f'}{G}:=\dfrac{1}{\left|G\right|}\sum_{g\in G}f(g)f'(g^{-1})\ \ (f,f'\in\MapC{G}).
\end{eqnarray}

\subsection{Formulations of operations of a \prelambdaring }\label{calculateresultoperate}
The statements about this section is described more detail in \cite{bu-1}. By using the inner product, we can write nonnegative integers $m_{i,j},n_{i,j}$ in  (\ref{approach}) as the following equalities.
\[ m_{i,j}=\inn{\chi_j}{S^i(\chi)}{G},\ n_{i,j}=\inn{\chi_j}{\lambda^i(\chi)}{G}. \]
for any $i=0,1,2,\dots$ and $j=1,2,\dots,k$. Hence the problem is the value of $S^i(\chi)(g)$ and $\lambda^i(\chi)(g)$ for $g\in G$, which is moved by an element of $G$ the character $S^i(\chi)$ and $\lambda^i(\chi)$. The reason is that the value of $S^i(\chi)(g)$ and $\lambda^i(\chi)(g)$ is used in (\ref{innerproduct}). Now, we describe the formula of operators of a \prelambdaring , Newtons formula and Propositon\ \ref{formulaanother} again. For all $n\in\mathbb{N}$, we have
\begin{eqnarray}
\label{Formula1} \psi^n-\lambda^1\psi^{n-1}+\cdots+(-1)^{n-1}\lambda^{n-1}\psi^1=(-1)^{n+1}n\lambda^n,\\
\label{Formula2} S^n+(-1)\lambda^1 S^{n-1}+\cdots+(-1)^{n-1}\lambda^{n-1}S^1+\lambda^n=0.
\end{eqnarray}
We apply the character $\chi$ to these equation, and move an element $g\in G$. We describe Adams operation on $\MapC{G}$ as follow, in Example \ref{MapGC}.
\begin{eqnarray}\label{Formula4}
\psi^n(f)(g)=f(g^n)\ (g\in G).
\end{eqnarray}
for $n\in\mathbb{N},\ f\in\MapC{G}$. Therefore using (\ref{Formula1}), (\ref{Formula2}) and (\ref{Formula4}), we can calculate $S^i(\chi)(g)$ and $\lambda^i(\chi)(g)$.

In addition, remarks that the character $S^i(\chi)$ and $\lambda^i(\chi)$ is a class function, which is satisfied $S^i(\chi)(ghg^{-1})=S^i(\chi)(h)$ and $\lambda^i(\chi)(ghg^{-1})=\lambda^i(\chi)(h)$ for any elements $g, h\in G$.

\begin{exa}[\cite{bu-1},\ p.95]\label{ExampleS31}
A symmetric group $S_3$ has only 2-dimension irreducible representation, so we consider about its 2nd  exterior power representation.

At first, we describe conjugate classes and the character table of $S_3$, which is the following statement and table.
\begin{eqnarray*}
C_1=\{ (1)\},\ C_2=\{ (1\ 2),(1\ 3),(2\ 3) \},\ C_3=\{ (1\ 2\ 3),(1\ 3\ 2)\}.
\end{eqnarray*}
\begin{center}
\begin{tabular}{|c||c|c|c|}\hline
          &$C_1$ & $C_2$ & $C_3$ \\ \hline\hline 
   $\chi_1$ &$1$&$ 1$&$1$\\ \hline
   $\chi_2$ &$1$&$-1$&$1$ \\ \hline
   $\chi_3$ &$2$ &$ 0$&$-1$ \\ \hline
\end{tabular}
\end{center}
It is clear that $\lambda^0(\chi_3)$ equals $\chi_1$ and $\lambda^1(\chi_3)$ equals $\chi_3$, we consider about $\lambda^2(\chi_3)$ by using Newton formula $\mathrm{(\ref{Formula1})}$ when $n=2$, we have
\[ \lambda^2(\chi_3)=\dfrac{1}{2}(\chi_3^2-\psi^2(\chi_3)). \]
Hence, $\mathrm{(\ref{Formula4})}$ and the table that noted earlier give
\begin{eqnarray*}
\lambda^2(\chi_3)((1))&=&\dfrac{1}{2}(2^2-2)=1, \\
\lambda^2(\chi_3)((1\ 2))&=&\dfrac{1}{2}(0^2-2)=-1, \\
\lambda^2(\chi_3)((1\ 2\ 3))&=&\dfrac{1}{2}((-1)^2-(-1))=1.
\end{eqnarray*}
Then, we obtain $\lambda^2(\chi_3)=\chi_2$. This finish to calculate about $\lambda^2(\chi_3)$, however, is the inner product or the like, we may require further calculations actually.
\end{exa}

\subsection{The generating function of multiplicity}\label{calculateresultgenerate}
To calculate the character $S^i(\chi)$ and $\lambda^i(\chi)$ for any $i\geq 0$, we consider $S_t(\chi),\lambda_t(\chi)$. Expanding each of the power series, we obtain
\begin{eqnarray}\label{approach1}
\begin{split}
S_t(\chi)=\sum_{i=0}^{\infty}S^i(\chi)t^i=\sum_{i=0}^{\infty}\Big(\sum_{j=1}^k m_{i,j}\chi_j\Big)t^i=\sum_{j=1}^k\Big(\sum_{i=0}^{\infty}\inn{\chi_j}{S^i(\chi)}{G}t^i\Big)\chi_j,\\
\lambda_t(\chi)=\sum_{i=0}^{\infty}\lambda^i(\chi)t^i=\sum_{i=0}^{\infty}\Big(\sum_{j=1}^k n_{i,j}\chi_j\Big)t^i=\sum_{j=1}^k\Big(\sum_{i=0}^{\infty}\inn{\chi_j}{\lambda^i(\chi)}{G}t^i\Big)\chi_j.
\end{split}
\end{eqnarray}
In this section, we define the generating function of the multiplicity of irreducible component.
\begin{defi}\label{genefunction}
We define the formal power series $\inn{\chi_j}{S_t(\chi)}{G}$ and $\inn{\chi_j}{\lambda_t(\chi)}{G}$ with coefficient is nonnegative integer for $j=1,\dots, k$ by the following equalities.
\begin{eqnarray}\label{lambdainn}
\begin{split}
\inn{\chi_j}{S_t(\chi)}{G}&:=\sum_{i=0}^{\infty}\inn{\chi_j}{S^i(\chi)}{G}t^i=\dfrac{1}{|G|}\sum_{g\in G} \chi_j(g)S_t(\chi)(g^{-1}), \\
\inn{\chi_j}{\lambda_t(\chi)}{G}&:=\sum_{i=0}^{\infty}\inn{\chi_j}{\lambda^i(\chi)}{G} t^i=\dfrac{1}{|G|}\sum_{g\in G} \chi_j(g)\lambda_t(\chi)(g^{-1}) \\ 
\end{split}
\end{eqnarray}
\end{defi}

By using (\ref{approach1}), we obtain the following equalities.
\begin{eqnarray}
S_t(\chi)=\sum_{j=1}^k \inn{\chi_j}{S_t(\chi)}{G} \chi_j,\ \lambda_t(\chi)=\sum_{j=1}^k \inn{\chi_j}{\lambda_t(\chi)}{G} \chi_j
\end{eqnarray}
A formal power series $\inn{\chi_j}{S_t(\chi)}{G}$ and $\inn{\chi_j}{\lambda_t(\chi)}{G}$ imply that the generating function of the multiplicity about $\chi_j$, hence it is believed that calculating such generating function is one means for computing the multiplicity. This method would be effective in calculating of the character of the symmetric powers representation especially of course in calculating of the character of the exterior powers representation.
\begin{exa}\label{ExampleS32}
We calculate about power series $S_t(\chi_3)$, that $\chi_3$ is 2-dimension irreducible representation of $S_3$. From $\mathrm{Example\ \ref{ExampleS31}}$, we obtain
\[ \lambda_t(\chi_3)=\chi_1+\chi_3t+\chi_2t^2 \]
which gives
\[ \lambda_t(\chi_3)((1))=(1+t)^2,\ \ \lambda_t(\chi_3)((1\ 2))=1-t^2,\ \ \lambda_t(\chi_3)((1\ 2\ 3))=1-t+t^2. \]
From $\mathrm{Definition\ \ref{defofsymm}}$, which is Definition of $\mathrm{Symmetric}$\ $\mathrm{powers}$\ $\mathrm{operation}$, we have
\begin{eqnarray}\label{calresult}
\begin{split}
S_t(\chi_3)((1))&=\dfrac{1}{(1-t)^2}=\sum_{i=0}\binom{i+1}{i}t^i,\\ 
S_t(\chi_3)((1\ 2))&=\dfrac{1}{1-t^2}=\sum_{i=0}t^{2i},\\ 
S_t(\chi_3)((1\ 2\ 3))&=\dfrac{1}{1+t+t^2}=\dfrac{1-t}{1-t^3}=\sum_{i=0}t^{3i}-t^{3i+1}
\end{split}
\end{eqnarray}
Putting these equalities into $\mathrm{(\ref{lambdainn})}$, we have
\begin{eqnarray}\label{calresult1}
\begin{split}
\inn{\chi_1}{S_t(\chi_3)}{S_3}&=\dfrac{1}{6}\Big(\dfrac{1}{(1-t)^2}+\dfrac{3}{1-t^2}+\dfrac{2}{1+t+t^2}\Big), \\
\inn{\chi_2}{S_t(\chi_3)}{S_3}&=\dfrac{1}{6}\Big(\dfrac{1}{(1-t)^2}-\dfrac{3}{1-t^2}+\dfrac{2}{1+t+t^2}\Big), \\
\inn{\chi_3}{S_t(\chi_3)}{S_3}&=\dfrac{1}{6}\Big(\dfrac{2}{(1-t)^2}-\dfrac{2}{1+t+t^2}\Big)
\end{split}
\end{eqnarray}
More calculating, we have another expressing
\begin{eqnarray}\label{calresult2}
\begin{split}
\inn{\chi_1}{S_t(\chi_3)}{S_3}&=\dfrac{1}{(1-t^2)(1-t^3)}, \\
\inn{\chi_2}{S_t(\chi_3)}{S_3}&=\dfrac{t^3}{(1-t^2)(1-t^3)}, \\
\inn{\chi_3}{S_t(\chi_3)}{S_3}&=\dfrac{t+t^2}{(1-t^2)(1-t^3)}=\dfrac{t}{(1-t)(1-t^3)}.
\end{split}
\end{eqnarray}

We have two form $\mathrm{(\ref{calresult1}) and (\ref{calresult2})}$ of the generating function of multiplicity, $\mathrm{(\ref{calresult1})}$ is oriented to see $S^n(\chi_3)$ for some $n\in\mathbb{N}$. For example, comparing the coefficients of $t^6$ in $\mathrm{(\ref{calresult}),\ (\ref{calresult1})}$, we have
\[ S^{6}(\chi_3)=\dfrac{\binom{7}{6}+3+2}{6}\chi_1+\dfrac{\binom{7}{6}-3+2}{6}\chi_2+\dfrac{\binom{7}{6}-1}{3}\chi_3=2\chi_1+\chi_2+2\chi_3\]
On the other hand, $\mathrm{(\ref{calresult2})}$ is more complexity than  $\mathrm{(\ref{calresult1})}$, however there is an advantage that if we have $\dfrac{1}{(1-t^2)(1-t^3)}$ in detail, we obtain the multiplicity for all $n\in\mathbb{N}$. It is believed that the change by the representation which one better in the calculation.

 We describe coefficients of $t^0, t^1, \dots, t^{10}$ in $\inn{\chi_1}{S_t(\chi_3)}{S_3},\ \inn{\chi_2}{S_t(\chi_3)}{S_3}$ and $\inn{\chi_3}{S_t(\chi_3)}{S_3}$ below. We focus on that the coefficient of $t^n$ in $\dfrac{1}{(1-t^2)(1-t^3)}$ is the number of the pair of nonnegative integer $(x,y)$ such that $2x+3y=n$.
\begin{center}
\begin{tabular}{|c||c|c|c|c|c|c|c|c|c|c|c|c|}\hline
          &$t^0$&$t^1$&$t^2$&$t^3$&$t^4$&$t^5$&$t^6$&$t^7$&$t^8$&$t^9$&$t^{10}$&$\cdots$ \\ \hline\hline 
   $\inn{\chi_1}{S_t(\chi_3)}{S_3}$&$1$&$0$&$1$&$1$&$1$&$1$&$2$&$1$&$2$&$2$&$2$&$\cdots$\\ \hline
   $\inn{\chi_2}{S_t(\chi_3)}{S_3}$&$0$&$0$&$0$&$1$&$0$&$1$&$1$&$1$&$1$&$2$&$1$&$\cdots$\\ \hline
   $\inn{\chi_3}{S_t(\chi_3)}{S_3}$&$0$&$1$&$1$&$1$&$2$&$2$&$2$&$3$&$3$&$3$&$4$&$\cdots$\\ \hline
\end{tabular}
\end{center}
\end{exa}

\subsection{An one-dimension representation}\label{calculateresult1st}
Let $\chi_1$ be a trivial character. In this section, we consider when $\chi$ has one dimension.

At first, if $i\geq 2$ then $\lambda^i(\chi)=0$. hence we have
\[ \lambda_t(\chi)=\chi_1+\chi t.\]
Next, we obtain
\[ S_t(\chi)=\dfrac{1}{\chi_1-\chi t}=\sum_{i=0}\chi^i t^i \]
which implies that $S^i(\chi)=\chi^i$ for any $i\in\mathbb{N}$. A set of all character of one dimension representation, which is a subset of $\MapC{G}$, has a structure of a group with the multiplication, so let $q$ is the order of $\chi$ about the multiplication, we have
\[ \langle \chi_j,S_t(\chi)\rangle_G=\dfrac{1}{1-t^q}\sum_{i=0}^{q-1} \delta_{\chi_j,\chi^i}t^i .\]

\begin{exa}\label{ExampleS33}
We consider about characters of one dimension representation of $S_3$, which are \ $\chi_1$ and $\chi_2$ (See $\mathrm{Example\ \ref{ExampleS31}}$). The order of these character is 1 and 2 respectively, we obtain
\begin{eqnarray*}
\lambda_t(\chi_1)&=&\chi_1+\chi_1t,\ \ S_t(\chi_1)=\dfrac{1}{1-t}\chi_1, \\
\lambda_t(\chi_2)&=&\chi_1+\chi_2t,\\ S_t(\chi_2)&=&\chi_1+\chi_2t+\chi_1t^2+\chi_2t^3+\cdots=\dfrac{1}{1-t^2}\chi_1+\dfrac{t}{1-t^2}\chi_2.
\end{eqnarray*}
\end{exa}

\subsection{A normal subgroup and the permutation representation }\label{calculateresultburnside}
In this section, we consider about the character which property is charactered by Adams operation,  and contain the character of the permutation representation. And as a specific example, we describe the symmetric and exterior powers representation of a  permutation representation which is generated by following action of $G$. $G$ acts on the set of left cosets $G/N$ of $N$ in $G$ by left multiplication $g(g'N):=(gg')N$, where $N$ is a normal subgroup of $G$.

 For any $g\in G$, we let $\left|g\right|$ be an order of an element $g$, and let $(a,b)$ be a greatest common divisor of integers $a$ and $b$. Since $G$ is a finite group, we have $\psi^{m_1}=\psi^{m_2}$ for any $m_1, m_2\geq 1$ with $m_1\equiv m_2\ \pmod {\left|G\right|}$. In addition, we define a function of $G$ which has more strict condition that is characterized by $|G|$.

\begin{defi}\label{AburnsideBdef}
We define $B$ by 
\[ B:=\{ f\in\MapC{G}\ \mid\ \psi^n(f)=\psi^{(n,|G|)}(f)\ \mathrm{for\ all}\ n=1,2,\dots,|G| \}. \]
It is a subset of $\MapC{G}$, 
\end{defi}

A set $B$ is closed by the addition from Proposition \ref{Adamshomo}. Actually, a set $B$ is a subring of $\MapC{G}$, and it is closed by $\lambda$-operation. However this proof use the properties of $\lambda$-ring, which is not described in this paper, and it is not use later. then, we do not describe this proof.

\begin{exa}\label{Exampleburnside}
Assume that a finite group $G$ acts a finite set $X$. Then, let $\chi$ be the character of the representation, whose representation space is generated by elements of $X$ as vector space. so, the character $\chi$ belongs to $B$.
\end{exa}
\begin{proof} For any $g\in G$, let $X_g$ be a set of an element $x\in X$ such that $gx=x$. Therefore we have $\chi(g)=\left|X_g\right|$, hence it suffices to show that $X_{g^n}=X_{g^{(n,|G|)}}$ for any $n\geq 1$.
 Let $x$ be an element of $X_{g^n}$. Then we have $g^nx=x$. Also, there are two integers $a$ and $b$ such that $an+b|G|=(n,|G|)$. We can choose an integer $a$ to $a\in\mathbb{N}$, hence we have $g^{(n,|G|)}x=g^{an+b|G|}x=g^{an}x=x$. Therefore we obtain $x\in X_g^{(n,|G|)}$. It is clear that $X_g^{(n,|G|)}\subset X_{g^n}$. This finish to prove $X_{g^n}=X_{g^{(n,|G|)}}$ for any $g\in G$.
\end{proof}

An element of $B$ has a such property.

\begin{theo}\label{Aburnside}
Let $a_1,a_2,\dots,a_r$ be all of divisor of $|G|$ such that $1=a_1< a_2 <\cdots < a_r=|G|$. For each $\chi\in B$, there are $b_1,\dots,b_r\in\MapC{G}$ uniquely, such that two following equalities hold.
\begin{eqnarray}\label{Aburnside1}
\lambda_t(\chi)=\prod_{i=1}^r (1-(-t)^{a_i})^{\frac{b_i}{a_i}}
\end{eqnarray}
and
\begin{eqnarray}\label{Aburnside2}
\psi^{a_l}(\chi)=\sum_{l';\ a_{l'}|a_l}b_{l'}
\end{eqnarray}
for any $l=1,\dots,r$.
\end{theo}
\begin{proof}  Since $G$ is a finite group, we have $\psi^{m_1}(\chi)=\psi^{m_2}(\chi)$ for any $m_1,m_2 \geq 1$ with $m_1\equiv m_2\ \pmod {\left|G\right|}$. Then, from the definition of Adams operation (Definition\ \ref{Adams0}), we have 
\begin{eqnarray}\label{Aburnsidechoiceproof1}
\begin{split}
-t\dfrac{d}{dt}\log\lambda_t(\chi)&=\sum_{i=1}\psi^i(\chi)(-t)^i \\
&=\Big(\dfrac{1}{1-(-t)^{|G|}}\Big)\Big(\sum_{i=1}^{|G|}\psi^i(\chi)(-t)^i \Big).
\end{split}
\end{eqnarray}
And more, let $T_l$ be the following equalities for any $l=1,2,\dots,r$.
\[ T_l:=\sum_{n=1,2,\dots,|G|,\ (n,|G|)=a_l} (-t)^{a_n-1} \]
Remark that $T_l$ belongs to $\mathbb{Z}[t]$. Since $\mathrm{(\ref{Aburnsidechoiceproof1})}$ and $\chi\in B$, we obtain
\begin{eqnarray}\label{Aburnsidechoiceproof2}
\dfrac{d}{dt}\log\lambda_t(\chi)=\Big(\dfrac{1}{1-(-t)^{|G|}}\Big)\Big(\sum_{i=1}^r \psi^{a_i}(\chi)T_i \Big).
\end{eqnarray}
Next, we let a polynomial $F_l$ be the following equalities for any $l=1,2,\dots,r$.
\[ F_l:=\sum_{i=l}^r A_{l,i}T_i\]
whose coefficient of $F_l$ is satisfied that following relations.
\begin{eqnarray}\label{Aburnsidechoiceproof3}
\begin{split}
F_1&:=\sum_{i=1}^r\psi^{a_i}(\chi)T_i \ \ (A_{1,i}=\psi^{a_i}(\chi)),\\
F_l&:=(1-(-t)^{|G|})\dfrac{d}{dt}\log(1-(-t)^{a_l})^{\frac{A_{l,l}}{a_l}}+F_{l+1}.
\end{split}
\end{eqnarray}
About this definition we calculate both side of second equality, then we have
\begin{eqnarray*}
&{}&F_l=A_{l,l}T_l+\sum_{i=l+1}^rA_{l,i}T_i,\\
&{}&(1-(-t)^{|G|})\dfrac{d}{dt}\log(1-(-t)^{a_l})^{\frac{A_{l,l}}{a_l}}+F_{l+1}\\
&=&A_{l,l}\dfrac{1-(-t)^{|G|}} {1-(-t)^{a_l}}+F_{l+1}\\
&=&A_{l,l}\big((-t)^{a_l-1}+(-t)^{2a_l-1}+\cdots+(-t)^{\frac{|G|}{a_l}a_l-1}\big)+\sum_{i=l+1}^rA_{l+1,i}T_i
\end{eqnarray*}
which implies that
\[ A_{l+1,i}=\begin{cases}A_{l,i} &(a_l\nmid a_i), \\ A_{l,i}-A_{l,l} &(a_l\mid a_i). \end{cases}\]
For each $A_{l,l}\ (\ l=1,2,\dots,r)$, let $i_1,\dots,i_p$ be numbers such that $a_i\mid a_l$ and $1=i_1<\cdots<i_p<a_l$. Then we have
\begin{eqnarray*}
A_{l,l}=A_{l-1,l}=\cdots=A_{i_p+1,l}&=&A_{i_p,l}-A_{i_p,i_p}\\
&=&\cdots=A_{i_1,l}-A_{i_1,i_1}-\cdots,A_{i_p,i_p} \\
&=&\psi^{a_l}(\chi)-A_{i_1,i_1}-\cdots,A_{i_p,i_p}
\end{eqnarray*}
which imply that
\[ \psi^{a_l}(\chi)=\sum_{l':\ a_{l'}|a_l}A_{l',l'}\]
In addition, calculating $F_r$ we obtain
\begin{eqnarray}\label{Aburnsidechoiceproof4}
F_r=A_{r,r}T_r=A_{r,r}(-t)^{|G|-1}=(1-(-t)^{|G|})\dfrac{d}{dt}\log(1-(-t)^{a_r})^{\frac{A_{r,r}}{a_r}}.
\end{eqnarray}
Therefore, by using $\mathrm{(\ref{Aburnsidechoiceproof2}),\ (\ref{Aburnsidechoiceproof3})}$ and $(\ref{Aburnsidechoiceproof4})$ we have
\begin{eqnarray*}
\dfrac{d}{dt}\log\lambda_t(\chi)&=&\dfrac{1}{1-(-t)^{|G|}}F_1 \\
&=&\dfrac{1}{1-(-t)^{|G|}}\Big((1-(-t)^{|G|})\dfrac{d}{dt}\log(1-(-t)^{a_1})^{\frac{A_{1,1}}{a_1}}+F_2 \Big) \\
&=&\cdots \\
&=&\dfrac{1}{1-(-t)^{|G|}}\sum_{i=1}^r(1-(-t)^{|G|})\dfrac{d}{dt}\log(1-(-t)^{a_i})^{\frac{A_{i,i}}{a_i}} \\
&=&\dfrac{d}{dt}\log\prod_{i=1}^r(1-(-t)^{a_i})^{\frac{A_{i,i}}{a_i}} \\
\end{eqnarray*}
which imply that
\[ \lambda_t(\chi)=\prod_{i=1}^r(1-(-t)^{a_i})^{\frac{A_{i,i}}{a_i}} \]
from the map $\dfrac{d}{dt}\log$ is injective. The uniquely is proved by (\ref{Aburnside2}) and the induction of $l=1,\dots, r$.

\end{proof}

We describe one of example which belongs to $B$. Let $N$ be a normal subgroup of $G$. $G$ acts on the quotient group $G/N$ that the following statement holds
\[ g\in G,\ \alpha=\pi(g')\ \Longrightarrow g\alpha:=\pi(gg') \]
where a map $\pi:G\rightarrow G/N$ is a natural projection. Let $\Pi$ is a character of a representation which generalized by above action. Then, $m\Pi$ belongs to $B$ for any $m\geq 1$ from Example\ \ref{Exampleburnside}. Now, we calculate $\inn{\chi_j}{S_t(m\Pi)}{G}$ and $\inn{\chi_j}{\lambda_t(m\Pi)}{G}$ for each $j=1,\dots, k$.

A character $\Pi$ is a induced representation from the trivial character from $N$ to $G$. Hence the specific values ​​is as follows.
\[ \Pi(g)=\begin{cases} |G/N| & (g\in N) \\ 0 &(g \notin N) \end{cases}\ (g\in G). \]

Now, we define the number $O_N(g)$ by the order of $\pi(g)\in G/N$.

\begin{lemm}\label{burnsideregularlemma}
We let $a_1,\dots,a_r$ be all of divisor of $\left|G\right|$ such that $1=a_1<a_2<\cdots<a_r=\left|G\right|$. About elements $b_1,\dots,b_r\in\MapC{G}$ in Theorem \ref{Aburnside}, the following equality holds for any $i=1,2,\dots,r$ and $g\in G$. 
\begin{eqnarray*}
b_i(g)=\begin{cases} m\left|G/N\right|&(O_N(g)=a_i)\\ 0&(O_N(g)\neq a_i) \end{cases}\ \ (g\in G).
\end{eqnarray*}
\end{lemm}
\begin{proof} We prove it by induction of $i=1,2,\dots,r$. If $i=1$, we have $b_1=\psi^1(m\Pi)=m\Pi$ from $\mathrm{(\ref{Aburnside1})}$, and $N$ is equal to a set of all element $g\in G$ such that $O_N(g)=1$. Hence we have this statement when $i=1$ from the value of $m\Pi$.

Suppose by induction for $j<i$. By calculating $b_i$ from $\mathrm{(\ref{Aburnside2})}$, we have
\[ b_i=\psi^{a_i}(m\Pi)-\sum_{l;\ a_l\mid a_i,\ l\neq i}b_l \]
Also, we obtain 
\[ \psi^{a_i}(m\Pi)(g)=\begin{cases} m|G/N| & (O_N(g)\mid a_i) \\ 0 &(O_N(g)\nmid a_i) \end{cases} \]
for all $g\in G$. Hence, we have
\begin{eqnarray*}
b_i(g)=\begin{cases} m|G/N| &(O_N(g)=a_i)\\ 0 &(O_N(g)\neq a_i) \end{cases}
\end{eqnarray*}
by the induction hypothesis. This finishes the induction and the proof of the Corollary.
\end{proof}

From the above Proposition and Corollary, we have the following Theorem.

\begin{theo}\label{burnsideregular}
The following three statement holds about the character $m\Pi$.
\begin{enumerate}
\item One has\ 
\begin{eqnarray}\label{burnsideregularproof1}
\begin{split}
 \lambda_t(m\Pi)(g)&=\big(1-(-t)^{O_N(g)}\big)^{\frac{m\left|G/N\right|}{O_N(g)}},\\ S_t(m\Pi)(g)&=\big(1-t)^{O_N(g)}\big)^{-\frac{m\left|G/N\right|}{O_N(g)}}.
\end{split}
\end{eqnarray}
for any $g\in G$.
\item One has
\begin{eqnarray}\label{burnsideregulareqn}
\begin{split}
\inn{\chi_j}{S_t(m\Pi)}{G}&=\dfrac{1}{\left|G\right|}\sum_{g\in G}\chi_j(g)\big(1-t^{O_N(g)}\big)^{-\frac{m|G/N|}{O_N(g)}}, \\
\inn{\chi_j}{\lambda_t(m\Pi)}{G}&=\dfrac{1}{\left|G\right|}\sum_{g\in G}\chi_j(g)\big(1-(-t\big)^{O_N(g)})^{\frac{m|G/N|}{O_N(g)}}
\end{split}
\end{eqnarray}
for any $j=1,\dots,k$.
\item In particular, if $n\geq 1$ is satisfied $(n,|G/N|)=1$, we have
\[S^n(m\Pi)=\dfrac{1}{|G/N|} \binom{m|G/N|+n-1}{n}\Pi,\ \ \lambda^n(m\Pi)=\dfrac{1}{|G/N|}\binom{m|G/N|}{n}\Pi \]
\end{enumerate}
\end{theo}
\begin{proof}$(1)$ It is clear by $\mathrm{Lemma\ \ref{burnsideregularlemma}}$.

$(2)$ We obtain $\mathrm{(\ref{burnsideregulareqn})}$ by apply $\mathrm{(\ref{burnsideregularproof1})}$ to $\mathrm{(\ref{lambdainn})}$.

$(3)$ About $\mathrm{(\ref{burnsideregulareqn})}$, we collect the number $O_N(g)$, we have
\begin{eqnarray}\label{burnsideregularproof2}
\begin{split}
\inn{\chi_j}{S_t(m\Pi)}{G}&=\dfrac{1}{\left|G\right|}\sum_{i=1}^r\Big(\sum_{O_N(g)=a_i}\chi_j(g)\Big)\big(1-t^{a_i}\big)^{-\frac{m|G/N|}{a_i}},\\ 
\inn{\chi_j}{\lambda_t(m\Pi)}{G}
&=\dfrac{1}{\left|G\right|}\sum_{i=1}^r\Big(\sum_{O_N(g)=a_i}\chi_j(g)\Big)\big(1-(-t)^{a_i}\big)^{\frac{m|G/N|}{a_i}}
\end{split}
\end{eqnarray}
Also, we obtain
\begin{eqnarray*}
\dfrac{1}{\left|G\right|}\sum_{O_N(g)=1}\chi_j(g)=\dfrac{1}{|G/N|}\inn{\chi_j}{\Pi}{G}.
\end{eqnarray*}

Assume that $n\geq 1$ is satisfied $(n,\ |G/N|)=1$. The multiplication of $\chi_j$ in $S^n(m\Pi)$ and $\lambda^n(m\Pi)$ are the coefficient of $t^n$ in each $\inn{\chi_j}{S_t(m\Pi)}{G}$ and $\inn{\chi_j}{\lambda_t(m\Pi)}{G}$. By $a_1=1$ and $O_N(g)$ is a divisor of $|G/N|$ for any $g\in G$, we have 
\begin{eqnarray*}
\inn{\chi_j}{S^n(m\Pi)}{G}&=&\dfrac{1}{|G|} \binom{m|G/N|+n-1}{n} \Big(\sum_{O_N(g)=a_i}\chi_j(g)\Big)\\ &=&\dfrac{1}{|G/N|}\binom{m|G/N|+n-1}{n}\inn{\chi_j}{\Pi}{G},\\
\inn{\chi_j}{\lambda^n(m\Pi)}{G}&=&\dfrac{1}{|G|} \binom{m|G/N|}{n} \Big(\sum_{O_N(g)=a_i}\chi_j(g)\Big)\\ &=&\dfrac{1}{|G/N|}\binom{m|G/N|}{n}\inn{\chi_j}{\Pi}{G}
\end{eqnarray*}
from $\mathrm{(\ref{burnsideregularproof2})}$. Hence we have\ $S^n(m\Pi)=\dfrac{1}{|G/N|} \binom{m|G/N|+n-1}{n}\Pi$ and $\lambda^n(m\Pi)=\dfrac{1}{|G/N|}\binom{m|G/N|}{n}\Pi$. 
\end{proof}

\begin{exa}
We calculate a character of regular representation of $S_3$, $\Pi$, which is obtained by putting $G=S_3$, $N=\{ e\}$ and $m=1$ in $\mathrm{Theorem\ \ref{burnsideregular}}$ About a character table of $S^3$, see $\mathrm{Example\ \ref{ExampleS31}}$. About the generating function of a multiplicity of the exterior powers representation, we have
\begin{eqnarray*}
\begin{split}
\inn{\chi_1}{\lambda_t(\Pi)}{S_3}&=\dfrac{1}{6}\Big((1+t)^6+3(1-t^2)^3+2(1+t^3)^2\Big),\\
\inn{\chi_2}{\lambda_t(\Pi)}{S_3}&=\dfrac{1}{6}\Big((1+t)^6-3(1-t^2)^3+2(1+t^3)^2\Big),\\ \inn{\chi_3}{\lambda_t(\Pi)}{S_3}&=\dfrac{1}{6}\Big(2(1+t)^6 -2(1+t^3)^2\Big).
\end{split}
\end{eqnarray*}
And about the generating function of a multiplicity of the symmetric powers representation, we have
\begin{eqnarray*}
\begin{split}
\inn{\chi_1}{S_t(\Pi)}{S_3}&=\dfrac{1}{6}\Big(\dfrac{1}{(1-t)^6}+\dfrac{3}{(1-t^2)^3}+\dfrac{2}{(1-t^3)^2}\Big), \\
\inn{\chi_2}{S_t(\Pi)}{S_3}&=\dfrac{1}{6}\Big(\dfrac{1}{(1-t)^6}-\dfrac{3}{(1-t^2)^3}+\dfrac{2}{(1-t^3)^2}\Big), \\
\inn{\chi_3}{S_t(\Pi)}{S_3}&=\dfrac{1}{6}\Big(\dfrac{2}{(1-t)^6}-\dfrac{2}{(1-t^3)^2}\Big),
\end{split}
\end{eqnarray*}
In addition we describe coefficient of $t^0,t^1,\dots t^{10}$ in these generating function.
\begin{center}
\begin{tabular}{|c||c|c|c|c|c|c|c|c|c|c|c|c|}\hline
          &$t^0$&$t^1$&$t^2$&$t^3$&$t^4$&$t^5$&$t^6$&$t^7$&$t^8$&$t^9$&$t^{10}$&$\cdots$ \\ \hline\hline 
$\inn{\chi_1}{\lambda_t(\Pi)}{S_3}$&$1$&$1$&$1$&$4$&$4$&$1$&$0$&$0$&$0$&$0$&$0$&$\cdots$\\ \hline
   $\inn{\chi_2}{\lambda_t(\Pi)}{S_3}$&$0$&$1$&$4$&$4$&$1$&$1$&$1$&$0$&$0$&$0$&$0$&$\cdots$\\ \hline
   $\inn{\chi_3}{\lambda_t(\Pi)}{S_3}$&$0$&$2$&$5$&$6$&$5$&$2$&$0$&$0$&$0$&$0$&$0$&$\cdots$\\ \hline\hline
   $\inn{\chi_1}{S_t(\Pi)}{S_3}$&$1$&$1$&$5$&$10$&$24$&$43$&$83$&$132$&$222$&$335$&$511$&$\cdots$\\ \hline
   $\inn{\chi_2}{S_t(\Pi)}{S_3}$&$0$&$1$&$2$&$10$&$18$&$43$&$73$&$132$&$207$&$335$&$490$&$\cdots$\\ \hline
   $\inn{\chi_3}{S_t(\Pi)}{S_3}$&$0$&$2$&$7$&$18$&$42$&$86$&$153$&$264$&$429$&$666$&$1001$&$\cdots$\\ \hline
\end{tabular}
\end{center}
\end{exa}

\subsection{A normal subgroup and its one-dimension representation}\label{calculateresultcentral}
Let $N$ be a normal subgroup of $G$ that contains a central of $G$ which is said to $Z(G)$. And let $\zeta$ be a character of a representation of $N$ which has 1 dimension. Now, we define $\zeta_0\in\MapC{G}$ as follows.
\[ \zeta_0(g)=\begin{cases} \zeta(g) & (g\in N), \\ 0 & (g\notin N). \end{cases} \]
We let $m\in\mathbb{N}$ be that $m\zeta_0$ is a character of $G$, and in this section, we consider about symmetric and exterior powers representation which has character $m\zeta_0$. The notation $O_N(g)$ and $(a,b)$, we used as those in the previous section.

Remarks that there is a such $m\in\mathbb{N}$. For example, if $m=|G/N|$ then $m\zeta_0$ is a induced character of $\zeta$ to $G$. In addition if $|G/N|$ equals to square number, then putting $m=\sqrt{|G/N|}$ gives $\inn{m\zeta_0}{m\zeta_0}{G}=1$. However, $m\zeta_0$ is not necessarily the representation.

\begin{theo}\label{centraltheorem}
The following three statement holds about  a character $m\zeta_0$.
\begin{enumerate}
\item For any $g\in G$, we have
\begin{eqnarray}\label{centrallemma2}
\lambda_t(m\zeta_0)(g)=\big(1-\zeta(g^h)(-t)^h \big)^{\frac{m}{h}} 
\end{eqnarray}
where $h=O_N(g)$.
\item One has
\begin{eqnarray}\label{centraltheoremeqn}
\begin{split}
\inn{\chi_j}{S_t(m\zeta_0)}{G}&=\dfrac{1}{\left|G\right|}\sum_{g\in G}\chi_j(g)\Big(1-\zeta(g^{-O_N(g)})t^{O_N(g)}\Big)^{-\frac{m}{O_N(g)}}\\
\inn{\chi_j}{\lambda_t(m\zeta_0)}{G}&=\dfrac{1}{\left|G\right|}\sum_{g\in G}\chi_j(g)\Big(1-\zeta(g^{-O_N(g)})(-t)^{O_N(g)}\Big)^{\frac{m}{O_N(g)}} \\
\end{split}
\end{eqnarray}
for any $j=1,\dots,k.$
\item In particular, if $n \geq 1$ is satisfied $(n,|G/N|)=1$, then we have
\[S^n(m\zeta_0)=\binom{m+n-1}{n}\zeta_0^n,\ \ \lambda^n(m\zeta_0)=\binom{m}{n}\zeta_0^n. \]

\end{enumerate}
\end{theo}

\begin{proof} $(1)$ By calculating Adams operation\ $\psi^n(m\zeta_0)(g)$ for any $n\geq 1$, we have 
\[ \psi^n(m\zeta_0)(g)=m\zeta_0(g^n)=\begin{cases}m\zeta(g^h)^i & (h\mid n,\ n=hi), \\ 0 & (h\nmid n)\end{cases} \]
since $\zeta$ has 1 dimension. Hence, we obtain
\begin{eqnarray}\label{centrallemma2proof1}
\sum_{i=1}^{\infty}m\zeta(g^h)^{i}(-t)^{hi}=-t\dfrac{d}{dt}\log \lambda_t(m\zeta_0)(g)
\end{eqnarray}
from the definition of Adams operation. Also, we have
\begin{eqnarray}\label{centrallemma2proof2}
-t\dfrac{d}{dt}\log (1-\zeta(g^h)(-t)^h)^{\frac{m}{h}}=\sum_{i=1}^{\infty}m\zeta(g^h)^{i}(-t)^{hi}
\end{eqnarray}
from $\mathrm{Lemma\ \ref{binomialdiff}}$. By (\ref{centrallemma2proof1}),\ (\ref{centrallemma2proof2}) and that a map $\dfrac{d}{dt}\log$ is $\mathrm{injective}$, this finish to prove $(1)$. 

$(2)$ We obtain $\mathrm{(\ref{centraltheoremeqn})}$ by apply $\mathrm{(\ref{centrallemma2})}$ to $\mathrm{(\ref{lambdainn})}$ for any $j=1,\dots,k$.

$(3)$ Let $a_1,\dots,a_n$ be all of a divisor of $\left|G\right|$ such that $1=a_1< a_2 <\cdots < a_r=\left|G\right|$. About $\mathrm{(\ref{centraltheoremeqn})}$ we collect the number $O_N(g)$, we have
\begin{eqnarray}\label{centraltheoremproof}
\begin{split}
\inn{\chi_j}{S_t(m\zeta_0)}{G}&=\dfrac{1}{\left|G\right|}\sum_{i=1}^r\Big(\sum_{O_N(g)=a_i}\chi_j(g)\big(1-\zeta(g^{-a_i})t^{a_i}\big)^{-\frac{m}{a_i}}\Big), \\
\inn{\chi_j}{\lambda_t(m\zeta_0)}{G}&=\dfrac{1}{\left|G\right|}\sum_{i=1}^r\Big(\sum_{O_N(g)=a_i}\chi_j(g)\big(1-\zeta(g^{-a_i})(-t)^{a_i}\big)^{\frac{m}{a_i}}\Big)
\end{split} 
\end{eqnarray}
Assume that $n\leq 1$ is satisfied $(n,|G/N|)=1$. By $a_1=1$ and $O_N(g)$ is a divisor of $|G/N|$ for any $g\in G$ , we have
\begin{eqnarray*}
\inn{\chi_j}{S^n(m\zeta_0)}{G}&=&\dfrac{1}{\left|G\right|}\sum_{g\in N}\chi_j(g)\binom{m+n-1}{n}\zeta^n(g^{-1})\\&=&\binom{m+n-1}{n}\inn{\chi_j}{\zeta_0^n}{G},\\
\inn{\chi_j}{\lambda^n(m\zeta_0)}{G}&=&\dfrac{1}{\left|G\right|}\sum_{g\in N}\chi_j(g)\binom{m}{n}\zeta^n(g^{-1})\\&=&\binom{m}{n}\inn{\chi_j}{\zeta_0^n}{G}.
\end{eqnarray*}
from $\mathrm{(\ref{centraltheoremproof})}$. Hence we have $S^n(m\zeta_0)=\binom{m+n-1}{n}\zeta_0^n,\ \lambda^n(m\zeta_0)=\binom{m}{n}\zeta_0^n$.
\end{proof}

\begin{exa}
Let $\zeta$ is a trivial character of $N$, and let $m$ be the multiple of |G/N|, which is written by $m'|G/N|$. then $m\zeta_0$ is a $m$ times the character of the representation which is generated by action $G$ to $G/N$. If $n\leq 1$ is satisfied $(n,|G/N|)=1$, we have
\begin{eqnarray*}
S^n(m'\Pi)&=&S^n(m\zeta_0)=\binom{m+n-1}{n}\zeta_0^n=\dfrac{1}{|G/N|}\binom{m'|G/N|+n-1}{n}\Pi ,\\
\lambda^n(m'\Pi)&=&\lambda^n(m\zeta_0)=\binom{m}{n}\zeta_0^n=\dfrac{1}{|G/N|}\binom{m'|G/N|}{n}\Pi 
\end{eqnarray*}
Calculating of this charcter $\Pi$ had also done in $\mathrm{Theorem\ \ref{burnsideregular}}$, However If $N\subset Z(G)$, then we can calculate by this method.
\end{exa}

\begin{exa}\label{ExampleHeisenberg}
We let $H_p$ be the following set.
\[ H_p:=\Bigl\{ \begin{pmatrix} 1 & a & b \\ 0 & 1 & c \\ 0 & 0 & 1 \end{pmatrix}\in GL(3,\mathbb{Z}/p\mathbb{Z})\ |\ a,b,c\in\mathbb{Z}/p\mathbb{Z} \Bigl\} 
\]
where $p$ is a prime number such that $p\neq 2$. It is said to Heisenberg group modulo $p$. It is a non commutative group, has cardinality $p^3$, and all of the element of $H_p$ except unit element has the order $p$.

Conjugate classes and character table of $H_p$ are the following statements and table where $\eta$ is a primitive $p$-th root of unity.
\begin{eqnarray*}
C(h)&:=&\Bigl\{\begin{pmatrix}1&0&h\\0&1&0\\0&0&1\end{pmatrix}\Bigl\}\ \ (h\in\mathbb{Z}/p\mathbb{Z}), \\
C(e,f)&:=&\Bigl\{\begin{pmatrix}1&e&k\\0&1&f\\0&0&1\end{pmatrix}\ \Bigl|\ k\in\mathbb{Z}/p\mathbb{Z}\Bigl\}\ \ (e,f\in\mathbb{Z}/p\mathbb{Z},\ (e,f)\neq(0,0)).
\end{eqnarray*}
\begin{center}
  \begin{tabular}{|c||c|c|} \hline
   &$C(h)$&$C(e,f)$ \\ \hline
\hline
   $\chi_{i,j}\ (i,j=0,1,\dots,p-1)$ &$1$&$\eta^{ei+fj}$\\ \hline
   $\tau_s\ (s=1,2,\dots,p-1)$ &$p\eta^{sh}$&$0$ \\ \hline
\end{tabular}
\end{center}

There are $p$ character which have one-dimension. The reason is that the quotient group $H_p/D(H_p)$ is isomorphic to the direct sum of two cycle group which has order $p$ where $D(H_p)$ is a commutator subgroup of $H_p$, which is the subgroup generated by the elements $xyx^{-1}y^{-1}$ for $x,y\in H_p$.   

We obtain the irreducible character which has dimension $p$ as follows. Let  $K$ be $K=\displaystyle\bigcup_{h=0}^{p-1}C(h) \displaystyle\cup \bigcup_{e=1}^{p-1}C(e,e)$, which is a subgroup of $H_p$ and is a direct sum of two cycle group which have order $p$ generated by two matrix
\[ \alpha=\begin{pmatrix} 1&1&0\\0&1&1\\0&0&1\end{pmatrix},\ \beta=\begin{pmatrix} 1&0&1\\0&1&0\\0&0&1\end{pmatrix}.\]
The character $\tau_s\ (s=1,2,\dots,p-1)$ is given by the induced character of following  character of $K$ to $H_p$.
\[ \gamma_s:K\rightarrow\mathbb{C},\ \gamma_s(\alpha^i\beta^j)=\eta^{js}\ (i,j=0,1,\dots,p-1).\ \]

 We consider the irreducible character $\tau_s$ for all $s=1,2,\dots,p-1$. At first, we define a map $\zeta_s:Z(H_p)\rightarrow \mathbb{C}$ as following while noted that $Z(H_p)=\displaystyle\bigcup_{h=0}^{p-1}C(h)$
\[ \zeta_s\Big(\begin{pmatrix} 1&0&h\\0&1&0\\0&0&1\end{pmatrix}\Big):=\eta^{hs}\ (h\in\mathbb{Z}/p\mathbb{Z}). \]
It is a character of $Z(H_p)$ which has one-dimension. In addition, we have $\tau_s=p({\zeta_s})_0$, hence let $n\geq 1$ be satisfied $(n,p)=1$, since $\mathrm{Theorem\ \ref{centraltheorem}}$ we obtain
\[ \lambda^n(\tau_s)=\binom{p}{n}(\zeta_s)_0^n=\dfrac{1}{p}\binom{p}{n}\tau_{n'},\ S^n(\tau_s)=\binom{p+n-1}{n}(\zeta_s)_0^n=\dfrac{1}{p}\binom{p+n-1}{n}\tau_{n'} \]
where $n'$ is satisfied that $n'=1,2,\dots,p-1$ and $ns\equiv n' \pmod p$.

Next we calculate $S^{ip}(\tau_s)$ and $\lambda^{ip}(\tau_s)$ for $i\geq 1$. Since $ \mathrm{(\ref{centrallemma2})}$, we have the following table.
\begin{center}
\begin{tabular}{|c||c|c|} \hline
&$C(h)$&$C(e,f)$ \\ \hline\hline
$\lambda_t(\tau_s)$ & $(1+\eta^{sh}t)^p$ & $1+t^p$ \\ \hline
$S_t(\tau_s)$ & $(1-\eta^{sh}t)^{-p}$ & $(1-t^p)^{-1} $ \\ \hline
\end{tabular}
\end{center}
About $\lambda^p(\tau_s)$, we obtain $\lambda^p(\tau_s)=\chi_{0,0}$ since the coefficient of $t^p$ in a polynomial which apply a element of $C(h)$ to $\lambda_t(\tau_s)$ is equal to $(\eta^{0h})(\eta^{1h})\cdots(\eta^{ph})=1$.\ Next, we calculate about $S^{ip}(\tau_s)$. If we apply each irreducible character of $H_p$ to $\mathrm{(\ref{centraltheoremproof})}$, we obtain as follows about $\chi_{i',j'}\ (i',j'=0,1,\dots,p-1)$.
\begin{eqnarray*}
& &\inn{\chi_{i',j'}}{S_t(\tau_s)}{H_p}\\
&=&\dfrac{1}{p^3}\Big(\sum_{h=0}^{p-1}(1-\eta^{sh}t)^{-p}+p\sum_{e,f,\ (e,f)\neq(0,0)}\eta^{i'e+j'f}(1-t^p)^{-1}\Big) \\
&=&\dfrac{1}{p^3}\Big(\sum_{h=0}^{p-1}(1-\eta^{sh}t)^{-p}-p(1-t^p)^{-1}+p\Big(\sum_{e=0}^p\eta^{i'e}\Big)\Big(\sum_{f=0}^p
\eta^{j'f}\Big)(1-t^p)^{-1}\\
&=&\begin{cases} \dfrac{1}{p^3}\big(\displaystyle\sum_{h=0}^{p-1}(1-\eta^{sh}t)^{-p}+p(p^2-1)(1-t^p)^{-1}\big) & ((i',j')=(0,0)) \\ \dfrac{1}{p^3}\big(\displaystyle\sum_{h=0}^{p-1}(1-\eta^{sh}t)^{-p}-p(1-t^p)^{-1}\big) & ((i',j')\neq(0,0)) \end{cases}
\end{eqnarray*}
then, the coefficient of $t^{ip}$ in $(1-\eta^{sh}t)^{-p}$ is $\binom{p+ip-1}{ip}(\eta^{sh})^{ip}=\binom{p+ip-1}{p}$ since $\mathrm{(\ref{inverseformula})}$. hence we  obtain
\begin{eqnarray}\label{centralexampleresult}
\begin{split}
S^{ip}(\tau_s)&=\dfrac{1}{p^2}\Big(\binom{p+ip-1}{ip}+(p^2-1)\Big)\chi_{0,0}\\&+\sum_{i,j\ (i,j)\neq(0,0)}\dfrac{1}{p^2}\Big(\binom{p+ip-1}{ip}-1\Big)\chi_{i,j}+\sum_{s'=1}^{p-1}a_{s'}\tau_{s'}
\end{split}
\end{eqnarray}
About the multiplicity of $\tau_s'\ (s'=1,2,\dots,p-1)$ which is $a_{s'}$, we obtain $0$ since calculating of dimension in the both side of $\mathrm{(\ref{centralexampleresult})}$.

To summarize the above, the following statement holds.
\begin{eqnarray*}
\lambda^n(\tau_s)&=&\begin{cases} \chi_{0,0} & (n=0,p), \\ \dfrac{1}{p}\ \displaystyle\binom{p}{n}\tau_{n'} & (n=1,\dots ,p-1),\\
0 & (n > p).
\end{cases}\\ 
S^n(\tau_s)&=&\begin{cases} \dfrac{1}{p^2}\Big(\ \displaystyle\binom{p+n-1}{n}+(p^2-1)\Big)\chi_{0,0}\\+\displaystyle\sum_{i,j\ (i,j)\neq(0,0)}\dfrac{1}{p^2}\Big(\binom{p+n-1}{n}-1\Big)\chi_{i,j} & (p \mid n), \\ \dfrac{1}{p}  \displaystyle\binom{p+n-1}{n}\tau_{n'} & (p\nmid n). 
\end{cases}
\end{eqnarray*}
for any $n\geq 1$ where $n'\geq 1 $ is satisfied that $n'=1,2,\dots,p-1$ and $ns\equiv n' \pmod p$.
\end{exa}

\subsection{Irreducible and reducible characters}\label{calculateresultirrirr}
For any given character $\chi$, we can write as follows.
\begin{eqnarray*}
S_t(\chi)&=&S_t(\inn{\chi_1}{\chi}{G}\chi_1+\inn{\chi_2}{\chi}{G}\chi_2+\cdots+\inn{\chi_k}{\chi}{G}\chi_k) \\
&=&S_t(\chi_1)^{\inn{\chi_1}{\chi}{G}}\cdots S_t(\chi_k)^{\inn{\chi_k}{\chi}{G}}, \\
\lambda_t(\chi)&=&\lambda_t(\inn{\chi_1}{\chi}{G}\chi_1+\inn{\chi_2}{\chi}{G}\chi_2+\cdots+\inn{\chi_k}{\chi}{G}\chi_k) \\
&=&\lambda_t(\chi_1)^{\inn{\chi_1}{\chi}{G}}\cdots \lambda_t(\chi_k)^{\inn{\chi_k}{\chi}{G}}.
\end{eqnarray*}
Therefore if we have the character of symmetric and exterior powers representation of all of the irreducible character, we obtain about the character of reducible representation too. Hence it will be important to calculate about the character of the  irreducible representation. However, we have to know the result of the multiplication of the irreducible character with each other to calculating, and it is unlikely that it will become easier in most cases. Thus, the method of described in previous section may be simpler even though the character of reducible representation.

\begin{exa}
We calculate about $\chi_1+\chi$ for any character $\chi$ where $\chi_1$ is a trivial character. We can apply about the natural representation of $S_n$ or $A_n$ which generated the action of $S_n$ or $A_n$ to a finite set $\{ 1,2,\dots, n \}$. Thus, we have

\begin{eqnarray*}
S_t(\chi_1+\chi)&=&\dfrac{1}{1-t}\chi_1S_t(\chi)=\dfrac{1}{1-t}S_t(\chi)
\end{eqnarray*}
which implies that $\inn{\chi_j}{S_t(\chi_1+\chi)}{G}=\dfrac{1}{1-t}\inn{\chi_j}{S_t(\chi)}{G}$. And we have
\[\lambda_t(\chi_1+\chi)=(\chi_1+\chi_1t)\big(\sum_{i=0}\lambda^i(\chi)t^i\big)=\chi_1+\sum_{i=1}(\lambda^{i-1}(\chi)+\lambda^i(\chi))t^i. \]
\end{exa}
\begin{exa}
Let $\chi$ be $\chi_1+\chi_2+\chi_3$ which is sum of all irreducible character of $S_3$. One has
\begin{eqnarray*}
\lambda_t(\chi)&=&\lambda_t(\chi_1)\lambda_t(\chi_2)\lambda_t(\chi_3) \\
&=&(\chi_1+\chi_1t)(\chi_1+\chi_2t)(\chi_1+\chi_3t+\chi_2t^2) \\
&=&\chi_1+(\chi_1+\chi_2+\chi_3)t+(\chi_2+\chi_2\chi_3+\chi_1\chi_3+\chi_1\chi_2)t^2\\
&+&(\chi_1\chi_2\chi_3+\chi_1\chi_2+\chi_2\chi_2)t^3+\chi_1\chi_2\chi_2t^4 \\
&=&\chi_1+(\chi_1+\chi_2+\chi_3)t+(2\chi_2+2\chi_3)t^2+(\chi_1+\chi_2+\chi_3)t^3+\chi_1t^4
\end{eqnarray*}
\end{exa}

\subsection{Calculating that use the quotient group}\label{calculateresultquotient}
Let $N$ is a normal subgroup of $G$. There is a case that the multiplicity in symmetric and exterior powers representation can compute by the quotient group $G/N$. Also in this section, let a map $\pi_N:G\rightarrow G/N$ be a natural projection. In addition, let  Cha($G$) and Cha($G/N$) be sets of all of character of $G$ and $G/N$ respectively. These sets are closed by the addition, the multiplication, $\lambda$-operation and symmetric powers operation. from $\S4$.

\begin{prop}\label{quotientpulback}
We define a map $\pi_N^*:\MapC{G/N}\rightarrow \MapC{G}$ as follows 
\[ \pi_N^*(f):=f\circ\pi\ (f\in\MapC{G/N}). \]
Then, the following statements hold.
\begin{enumerate}
\item A map $\pi_N^*$ is a $\lambda$-monomorphism.
\item We have $\inn{f_1}{f_2}{G/N}=\inn{\pi_N^*(f_1)}{\pi_N^*(f_2)}{G}$ for all $f_1,f_2\in\MapC{G/N}.$
\end{enumerate}
\end{prop}

\begin{proof} $(1)$ It is clear that a map $\pi_N^*$ is a ring\ monomorphism, then we  show that $\lambda^n\circ\pi_N^*=\pi_N^*\circ\lambda^n$ for any $n\geq 1$ $\MapC{G}$ which is a codomain of $\pi_N^*$, is the $\mathbb{Z}$-torsion-free, then it is suffieicnt to show that $\psi^n\circ\pi_N^*=\pi_N^*\circ\psi^n$ for any $n\geq 1$ from Proposition\ \ref{lambdafcha2}. For any $f\in\MapC{G/N}, n\geq 1$ and $g\in G$. we have
\[ \psi^n(\pi_N^*(f))(g)=\pi_N^*(f)(g^n)=f\circ\pi(g^n)=f(\pi(g)^n)=\psi^n(f)(\pi(g))=\pi_N^*(\psi^n(f))(g).\]
Hence, A ring monomorphism $\pi_N^*$ is a $\lambda$-monomorphism. 

$(2)$\ For any $f_1,f_2\in\MapC{G/N}$, we calculate the left hand side as follow.
\begin{eqnarray*}
\inn{\pi_N^*(f_1)}{\pi_N^*(f_2)}{G}
&=&\dfrac{1}{\left|G\right|}\sum_{g\in G}\pi_N^*(f_1)(g)\pi_N^*(f_2)(g^{-1})\\
&=&\dfrac{1}{\left|G\right|}\sum_{g\in G}f_1(\pi(g))f_2(\pi(g)^{-1})\\
&=&\dfrac{|N|}{\left|G\right|}\sum_{\alpha\in G/N}f_1(\alpha)f_2(\alpha^{-1})=\inn{f_1}{f_2}{G/N}.
\end{eqnarray*}
which imply to finish to prove $(2)$.
\end{proof}

\begin{lemm}\label{quotientlemma}
Let $\chi$ a character of a repressentation $\rho:G\rightarrow GL(V)$. we have  following statements.
\begin{enumerate}
\item $\mathrm{Ker}(\rho)=\{ g\in G\ \mid\ \chi(g)=\chi(e) \}$.
\item $N\subset\mathrm{Ker}(\rho)$ holds if only and if $\chi(g)=\chi(e)$ holds for any $g\in N$.
\end{enumerate}
\end{lemm}
\begin{proof} $(1)$ For any $g\in \mathrm{Ker}(\rho)$, we have $\rho(g)=\mathrm{id}_V$. hence $\chi(g)=\chi(e)$ holds.

 Let $n$ be $\dim V$, and let $\omega$ is a primitive $|G|$-th root of unity. For each $g\in G$, since a representation $\rho$ is a unitary, then there are integer numbers $a_1,\dots,a_n$ such that $\chi(g)=\omega^{a_1}+\cdots+\omega^{a_n}$. Assume that $\chi(g)=\chi(e)=n$ the we have $a_1,\dots,a_n=0$ which imply that all eigenvalue of $\rho(g)$ are $1$, then $\rho(g)=\mathrm{id}_V$ holds. Hence, we have $g\in \mathrm{Ker}(\rho)$. At last, it is clear that $(2)$ holds from $(1)$. 
\end{proof}

\begin{prop}\label{quotientimage}
We define $F(N)$ as follows.
\[ F(N):=\{ \chi\in\mathrm{Cha}(G)\ \mid\ \chi(g)=\chi(e)\ \mathrm{for\ all}\ g\in N \} \]
This set is a character of  a representation $\rho:G\rightarrow GL(V)$ which is satisfied  $N\subset \mathrm{Ker}(\rho)$, from $\mathrm{Lemma\ \ref{quotientlemma}\ (2)}$. About this, $F(N)=\pi_N^*(\mathrm{Cha}(G))$ holds.
\end{prop}
\begin{proof}
Let $\chi$ be the character of $F(N)$, and let $\rho$ be the representation of $G$ such that its character is $\chi$. Since $N\subset\mathrm{Ker}(\rho)$, There is a representation $\rho'$ of $G/N$ uniquely such that $\rho=\rho'\circ\pi$ holds from $\mathrm{Lemma\ \ref{quotientlemma}}$. Let $\chi'$ be a character of $\rho'$, we have $\chi=\chi'\circ\pi_N=\pi_N^*(\chi')\in \pi_N^*(\mathrm{Cha}(G/N))$.

Conversely, assume that $\chi$ belongs to $\pi_N^*(\mathrm{Cha}(G/N))$, and let $\chi'$ be $\chi=\pi_N^*(\chi')$. Then we have $\chi(g)=\pi_N^*(\chi')(g)=\chi'(e)=\chi(e)$ for any $g\in N$. Also we let $\rho':\ G/N\rightarrow GL(V)$ the representation whose character is $\chi'$. Then we have
\[ \chi=\pi_N^{*}(\chi')=\chi'\circ\pi=\mathrm{Tr}\circ(\rho'\circ\pi) \]
which implies that $\chi$ is a charcter of $\rho'\circ\pi$. Hence we have $\chi\in F(N)$.
\end{proof}

\begin{cor}\label{quotientcorollary}
The following statements holds for $f\in\MapC{G/N}$.
\begin{enumerate}
\item A map $f$ belongs to $\mathrm{Cha}(G/N)$ if and only if $\pi_{N}^*(f)$ belongs to $\mathrm{Cha}(G)$. 
\item A map $f$ is a irreducible character of $G/N$ if and only if $\pi_{N}^*(f)$ is a irreducible character of $G$.
\end{enumerate}
\end{cor}
\begin{proof}$(1)$ is clear from Proposition\ \ref{quotientimage}, and $(2)$ is given by $(1)$ and Proposition\ \ref{quotientpulback}\ (2).
\end{proof}

Let $\phi_1,\dots,\phi_{k'}$ be all of irreducible character of $G/N$. Since Corollary\ \ref{quotientcorollary}\ (2), $\pi_N^*(\phi_1),\dots,\pi_N^*(\phi_{k'})$ are irreducible character of $G$ Hence we can be written $\mathrm{Cha}$($G/N$) and image of $\pi_N^*$, $F(N)$ as follows. 
\begin{eqnarray*}
&{}&\mathrm{Cha}(G/N)=\{ n_1\phi_1+\cdots+n_{k'}\phi_{k'}\ \mid\ n_1,\dots,n_{k'}\in\mathbb{Z}_+\} \\
&{}&F(N)=\{ n_1\pi_N^*(\phi_1)+\cdots+n_{k'}\pi_N^*(\phi_{k'})\ \mid\ n_1,\dots,n_{k'}\in\mathbb{Z}_+\}.
\end{eqnarray*}

\begin{theo}\label{quotienttheorem}
Suppose that the character belongs to $F(N)$ and the character $\chi'\in\mathrm{Cha}(G/N)$ is satisfied that $\chi=\pi_N^*(\chi')$ holds. Thus, the following statements holds.
\begin{enumerate}
\item Character $\lambda^i(\chi)$ and $S^i(\chi)$ belong to $F(N)$ for any $i\geq 1$.
\item For any $i=0,1,2,\dots$ and $j=1,2,\dots, k'$, we have
\[ \inn{\pi_N^*({\phi_j})}{S^i(\chi)}{G}=\inn{\phi_j}{S^i(\chi')}{G/N},\ \inn{\pi_N^*({\phi_j})}{\lambda^i(\chi)}{G}=\inn{\phi_j}{\lambda^i(\chi')}{G/N}.\]
\end{enumerate}
\end{theo}
\begin{proof}
$(1)$ In $\mathrm{Proposition\ \ref{quotientpulback}\ (1)}$, we proved that a map $\pi_N^*$ can be exchanged $\lambda$-operation. From $\mathrm{Proposition\ \ref{lambdafcha2}}$, it implies that $\pi_N^*$ can be exchanged symmetric\ powers\ operation too. Therefore we have
\begin{eqnarray*}\label{quotienttheoremproof1}
\lambda^i(\chi)&=&\lambda^i\circ\pi_N^*(\chi')=\pi_N^*\circ\lambda^i(\chi')\in F(N), \\
S^i(\chi)&=&S^i\circ\pi_N^*(\chi')=\pi_N^*\circ S^i(\chi')\in F(N)
\end{eqnarray*}
for any $i\geq 0$.

$(2)$ Since $\pi_N^*$ can be exchanged $\lambda$-operation and symmetric powers operation we have the following equality from $\mathrm{Proposition\ \ref{quotientpulback}\ (2)}$ for any $i=0,1,2,\dots$ and $j=1,2,\dots,k'$. 
\begin{eqnarray*}
\inn{\pi_N^*(\phi_j)}{\lambda^i(\chi)}{G}&=&\inn{\pi_N^*(\phi_j)}{\lambda^i\circ\pi_N^*(\chi')}{G} \\
&=&\inn{\pi_N^*(\phi_j)}{\pi_N^*\circ\lambda^i(\chi')}{G}=\inn{\phi_j}{\lambda^i(\chi') }{G/N}.
\end{eqnarray*}
Similarly, we have $\inn{\pi_N^*(\chi_j)}{S^i(\chi)}{G}=\inn{\chi_j}{S^i(\chi')}{G/N}$.
\end{proof}

In summary, we can be calculated about a character $\chi$ which belongs to $F(N)$ 
by the character $\chi'$ of $G/N$. In particular, any character $\chi$ which is the character of the representation $\rho$ belongs to $\mathrm{Ker}(\rho)$, then \ we can consider on the representation of $G/\mathrm{Ker}(\rho)$. In addition , this representation is faithful, hence we can assume the faithfully of the representation when we calculate the multiplicity.

\begin{exa}
Let $\Pi$ is a character which is generated by action $G$ to $G/N$. $\Pi$ belongs to $F(N)$, and correspond to a character of a regular representation of $G/N$.
\end{exa}

\begin{exa}\label{ExampleS4}
We consider irreducible characters of $S_4$ which has dimension 1 or 2. Conjugate classes and irreducible character which has dimension 1 or 2 is the following statement and table.
\begin{eqnarray*}
(1)\in C_1,\ (1\ 2)\in C_2,\ (1\ 2\ 3)\in C_3,\ (1\ 2\ 3\ 4)\in C_4,\ (1\ 2)(3\ 4)\in C_5.
\end{eqnarray*}
\begin{center}
  \begin{tabular}{|c||c|c|c|c|c|} \hline
   &$C_1$&$C_2$&$C_3$&$C_4$&$C_5$ \\ \hline\hline
   $\phi_1$ &$1$&$1$&$1$&$1$&$1$ \\ \hline
   $\phi_2$ &$1$ &$-1$&$1$&$-1$&$1$ \\ \hline
   $\phi_3$ &$2$ &$0$ &$-1$ &$ 0$ &$ 2$  \\ \hline
  \end{tabular}
\end{center}
Now, we let $V$ is $C_1 \cup C_5$. Thus $V$ is a normal subgroup of $S_4$, and the quotient group $S_4/V$ is isomorphic to $S_3$. And $\phi_1,\phi_2,\phi_3\in F(V)$ holds, which correspond to $\chi_1, \chi_2, \chi_3$ by$\pi_V^*:\MapC{S_3}\rightarrow F(V)$ where characters $\chi_1, \chi_2, \chi_3$ are irreducible character of $S_3$ (See  $\mathrm{Example\ \ref{ExampleS31}}$), Therefore the character sum of some $\phi_1,\phi_2,\phi_3$ can be considered in $S_3$. In particular, we can calculate about $\phi_3$ from $\mathrm{Example\ \ref{ExampleS31} and Example\ \ref{ExampleS32}}$, then we have
\begin{eqnarray*}
\lambda_t(\phi_3)&=&\phi_1+\phi_3t+\phi_2t^2, \\
S_t(\phi_3)&=&\dfrac{1}{(1-t^2)(1-t^3)}\phi_1+\dfrac{t^3}{(1-t^2)(1-t^3)}\phi_2+\dfrac{t+t^2}{(1-t^2)(1-t^3)}\phi_3.
\end{eqnarray*}
\end{exa}

\newpage

\section{Calculation of some finite groups}\label{somecalculate}
In this chapter, we calculate characters of the symmetric powers representation and the exterior powers representation of some finite group, about symmetric group $S_3$ and $S_4$, alternative group $A_4$ and $A_5$, non-commutative group that has cardinality  of 21 $G_{21}$, the dihedral groups $D_{2n}$, the generalized quaternion groups $Q_{4n}$ and Heisenberg groups $H_p$.
 
 At each some finite group, we give the conjugate class and character table. Next, we give the conjugate class where a inverse element of its belongs, a cardinality of conjugate class and a order of a element of conjugate class. These information of a finite group are used on calculate about inner product. 
 
 We calculate for only the irreducible representation with has dimension more than $2$ is  here. The reason is because if we have the result about irreducible representations, we can calculate about reducible representations from $\S\ref{calculateresultirrirr}$. In addition, The result of one dimension representation was given in \ref{calculateresult1st}, so we keep it to  only give a order of the multiplication as a element of  $\MapC{G}$. Some finite group and its character was described in $\S\ref{calculateresult1}$, however we will write again.

Basically, the method of calculation about $S_3$, $A_4$, $G_{21}$, $S_4$ and $A_5$, it will describe in the chapter $\S\ref{resultS3}$ to $\S\ref{resultA5}$, is the same as Example\ \ref{ExampleS31} and Example\ \ref{ExampleS32} At first, we calculate the character of exterior powers representation using by (\ref{Formula1}) and (\ref{Formula4}), next we calculate $\lambda_t(\chi)(g)$ and $S_t(\chi)(g)$ for each $g\in G$, and compute the generating function of multiplicity of irreducible component  (\ref{lambdainn}) for any irreducible character.

\subsection{A symmetric group $S_3$}\label{resultS3}
We have already said about $S_3$ in Example \ref{ExampleS31} and Example\ \ref{ExampleS32}.

Conjugate classes and the character table of $S_3$ is the following statement and table.
\begin{eqnarray*}
C_1=\{ (1)\},\ C_2=\{ (1\ 2),(1\ 3),(2\ 3) \},\ C_3=\{ (1\ 2\ 3),(1\ 3\ 2)\}.
\end{eqnarray*}
\begin{center}
\begin{tabular}{|c||c|c|c|c|}\hline
     &$C_1$&$C_2$&$C_3$&Remarks\\ \hline
\hline   
   Inverse &$C_1$&$C_2$&$C_3$& \\ \hline
   Cardinality&$1$&3&$2$& \\ \hline   
   Order&$1$&$ 2$&$3$& \\  \hline
\hline
   $\chi_1$ &$1$&$ 1$&$1$&It is a trivial.\\ \hline
   $\chi_2$ &$1$&$-1$&$1$&It has order $2$. \\ \hline
   $\chi_3$ &$2$ &$ 0$&$-1$& \\ \hline
\end{tabular}
\end{center}
\begin{itemize}
\item About a character $\chi_3$

\begin{tabular}{|c||c|c|c|}\hline
   &$C_1$&$C_2$&$C_3$\\ \hline\hline
   $\chi_3$ &$2$& $0$ &$-1$\\ \hline
   $\lambda_t(\chi_3)$ &$(1+t)^2$ &$1-t^2$&$1-t+t^2$ \\ \hline
   $S_t(\chi_3)$ &$\dfrac{1}{(1-t)^2}$ & $\dfrac{1}{1-t^2}$ & $\dfrac{1}{1+t+t^2}$ \\ \hline
\end{tabular}
\begin{eqnarray*}
\lambda_t(\chi_3)&=&\chi_1+\chi_3 t+\chi_2 t^2. \\
S_t(\chi_3)&=&\dfrac{1}{(1-t^2)(1-t^3)}\chi_1+\dfrac{t^3}{(1-t^2)(1-t^3)}\chi_2+\dfrac{t}{(1-t)(1-t^3)}\chi_3 \\
&=&\dfrac{1}{(1-t^2)(1-t^3)}\Big(\chi+t^3\chi_2+(t+t^2)\chi_3\Big).
\end{eqnarray*}

\end{itemize}

\subsection{An alternating group $A_4$}\label{resultA4}

Conjugate classes and the character table of $A_4$ is the following statement and table where $\omega$ is a primitive $3$-th root of unity.

\begin{eqnarray*}
(1)\in C_1,\ (1\ 2)(3\ 4)\in C_2,\ (1\ 2\ 3)\in C_3,\ (1\ 3\ 2)\in C_4.
\end{eqnarray*}
\begin{center}
\begin{tabular}{|c||c|c|c|c|c|}\hline
               &$C_1$&$C_2$&$C_3$&$C_4$&Remarks\\ \hline
\hline
   Inverse&$C_1$&$C_2$&$C_4$&$C_3$&\\ \hline   
   Cardinality& $1$& $3$ & $4$ & $4$ & \\ \hline
   Order&$1$ & $2$ &$3$ &$ 3$& \\ \hline
\hline
   $\chi_1$ &$1$ & $1$ &$1$ &$ 1$& It is a trivial. \\ \hline
   $\chi_2$ &$1$ & $1$ &$\omega$ &$\omega^2$& It has order 3.\\ \hline
   $\chi_3$ &$1$ & $1$ &$\omega^2$& $\omega$& It has order 3. \\ \hline
   $\chi_4$ &$3$ & $-1$& $0$& $0$ &\\ \hline
\end{tabular}
\end{center}

\begin{itemize}
\item About a character $\chi_4$

\begin{tabular}{|c||c|c|c|c|}\hline
               &$C_1$&$C_2$&$C_3$&$C_4$\\ \hline\hline
   $\chi_4$ &$3$& $-1$ &$0$ &$0$\\ \hline
   $\lambda_t(\chi_4)$ &$(1+t)^3$ &$(1-t)(1-t^2)$&$1+t^3$&$1+t^3$ \\ \hline
   $S_t(\chi_4)$ &$\dfrac{1}{(1-t)^3}$ & $\dfrac{1}{(1+t)(1-t^2)}$ & $\dfrac{1}{1-t^3}$& $\dfrac{1}{1-t^3}$ \\ \hline
\end{tabular}

\begin{eqnarray*}
\lambda_t(\chi_4)&=&\chi_1+\chi_4 t+\chi_4t^2 +\chi_1 t^3. \\
S_t(\chi_4)&=&\dfrac{1-t^2+t^4}{(1-t^2)^2(1-t^3)}\chi_1+\dfrac{t^2}{(1-t^2)^2(1-t^3)}\chi_2\\
&+&\dfrac{t^2}{(1-t^2)^2(1-t^3)}\chi_3+\dfrac{t}{(1+t)^2(1-t^3)}\chi_4 \\
&=&\dfrac{1}{(1-t^2)^2(1-t^3)}\Big((1-t^2+t^4)\chi_1+t^2\chi_2+t^2\chi_3+(t+t^2+t^3)\chi_4\Big).
\end{eqnarray*}

\end{itemize}

\subsection{A non-commutative group $G_{21}$ which has order 21}\label{resultG21}
We let $G_{21}$ be a finite group that is generated by two elements $x$ and $y$ which is satisfied relations $x^7=y^3=e$ and $y^{-1}xy=x^2$. This group has order 21, which has the most small of odd cardinality in the non-commutative group.

Conjugate classes and the character table of $G_{21}$ is the following statement and table where $\omega$ is a primitive $3$-th root of unity and $\eta$ is a primitive $7$-root of unity. In addition, we let the number $a$ and $b$ be satisfied equalities $a:=\eta+\eta^2+\eta^4=(-1+i\sqrt{7})/2$ and $b:=\eta^3+\eta^5+\eta^6=(-1-i\sqrt{7})/2$. we have $a+b=-1$ and $ab=2.$
\begin{eqnarray*}
C_1=\{ e\},\ C_2=\{ x,x^2,x^4 \},\ C_3=\{ x^3,x^5,x^6\},\\
C_4=\{ x^i y\ |\ i=0,\dots,5 \},\ C_5=\{ x^i y^2\ |\ i=0,\dots,5 \}.
\end{eqnarray*}
\begin{center}
\begin{tabular}{|c||c|c|c|c|c||c|}\hline
               &$C_1$&$C_2$&$C_3$&$C_4$&$C_5$&Remarks\\ \hline
\hline
   Inverse&$C_1$&$C_3$&$C_2$&$C_5$&$C_4$& \\ \hline   
   Cardinality&$1$ &$3$&$3$&$7$&$7$&\\ \hline
   Order&$1$ &$7$ &$7$&$3$ &$3$& \\ \hline
\hline
   $\chi_1$ &$1$ &$1$ &$1$&$1$ &$1$& It is a trivial, and has order 1. \\ \hline
   $\chi_2$ &$1$ &$1$& $1$& $\omega$ & $\omega^2$&It has order 3. \\ \hline
   $\chi_3$ &$1$ &$1$& $1$ & $\omega^2$ & $\omega$&It has order 3. \\ \hline
   $\chi_4$ & $3$ &$a$&$b$& $0$ &$0$& \\ \hline
   $\chi_5$ & $3$ &$b$&$a$&$0$ & $0$& \\ \hline
\end{tabular}
\end{center}

\begin{itemize}

\item About a character $\chi_4$ \\

\begin{tabular}{|c||c|c|c|c|c|}\hline
               &$C_1$&$C_2$&$C_3$\\ \hline\hline
   $\chi_4$ & 3 &$a$&$b$\\ \hline
   $\lambda_t(\chi_4)$ &$(1+t)^3$ &$1+at+bt^2+t^3$&$1+bt+at^2+t^3$\\ \hline
   $S_t(\chi_4)$ &$\dfrac{1}{(1-t)^3}$ & $\dfrac{1}{1-at+bt^2-t^3}$&$\dfrac{1}{1-bt+at^2-t^3}$ \\ \hline
\end{tabular}

\begin{tabular}{|c||c|c|c|c|c|}\hline
               &$C_4$&$C_5$\\ \hline\hline
   $\chi_4$ & 0 &0 \\ \hline
   $\lambda_t(\chi_4)$ &$1+t^3$&$1+t^3$ \\ \hline
   $S_t(\chi_4)$  & $\dfrac{1}{1-t^3}$& $\dfrac{1}{1-t^3}$ \\ \hline
\end{tabular}

\begin{eqnarray*}
\lambda_t(\chi_4)&=&\chi_1+\chi_4 t+\chi_5 t^2+\chi_1 t^3. \\
S_t(\chi_4)&=&\dfrac{1-t+t^4-t^7+t^8}{(1-t)(1-t^3)(1-t^7)}\chi_1+\dfrac{t^4}{(1-t)(1-t^3)(1-t^7)}\chi_2\\
&+&\dfrac{t^4}{(1-t)(1-t^3)(1-t^7)}\chi_3+\dfrac{t-t^2+t^4}{(1-t)^2(1-t^7)}\chi_4\\
&+&\dfrac{t^2-t^4+t^5}{(1-t)^2(1-t^7)}\chi_5 \\
&=&\dfrac{1}{(1-t)(1-t^3)(1-t^7)}\Big((1-t+t^4-t^7+t^8)\chi_1+t^4\chi_2+t^4\chi_3\\
&+&(t+t^5+t^6)\chi_4+(t^2+t^3+t^7)\chi_5\Big).
\end{eqnarray*}

\item About a character $\chi_5$ \\

\begin{tabular}{|c||c|c|c|c|c|}\hline
               &$C_1$&$C_2$&$C_3$\\ \hline\hline
   $\chi_4$ & 3 &$b$&$a$\\ \hline
   $\lambda_t(\chi_4)$ &$(1+t)^3$ &$1+bt+at^2+t^3$&$1+at+bt^2+t^3$\\ \hline
   $S_t(\chi_3)$ &$\dfrac{1}{(1-t)^3}$ & $\dfrac{1}{1-bt+at^2-t^3}$&$\dfrac{1}{1-at+bt^2-t^3}$ \\ \hline
\end{tabular}

\begin{tabular}{|c||c|c|c|c|c|}\hline
               &$C_4$&$C_5$\\ \hline\hline
   $\chi_4$ & 0 &0 \\ \hline
   $\lambda_t(\chi_4)$ &$1+t^3$&$1+t^3$ \\ \hline
   $S_t(\chi_3)$ & $\dfrac{1}{1-t^3}$& $\dfrac{1}{1-t^3}$ \\ \hline
\end{tabular}

\begin{eqnarray*}
\lambda_t(\chi_5)&=&\chi_1+\chi_5 t+\chi_4 t^2+\chi_1 t^3. \\
S_t(\chi_4)&=&\dfrac{1-t+t^4-t^7+t^8}{(1-t)(1-t^3)(1-t^7)}\chi_1+\dfrac{t^4}{(1-t)(1-t^3)(1-t^7)}\chi_2\\
&+&\dfrac{t^4}{(1-t)(1-t^3)(1-t^7)}\chi_3+\dfrac{t^2-t^4+t^5}{(1-t)^2(1-t^7)}\chi_4\\
&+&\dfrac{t-t^2+t^4}{(1-t)^2(1-t^7)}\chi_5 \\
&=&\dfrac{1}{(1-t)(1-t^3)(1-t^7)}\Big((1-t+t^4-t^7+t^8)\chi_1+t^4\chi_2+t^4\chi_3\\
&+&(t^2+t^3+t^7)\chi_4(t+t^5+t^6)\chi_5\Big).
\end{eqnarray*}
\end{itemize}

\subsection{A symmetric group $S_4$}\label{resultS4}
Conjugate classes and the character table of $S_4$ is the following statement and table (\cite{bu-2}\ p.287).
\begin{eqnarray*}
(1)\in C_1,\ (1\ 2)\in C_2,\ (1\ 2\ 3)\in C_3,\ (1\ 2\ 3\ 4)\in C_4,\ (1\ 2)(3\ 4)\in C_5.
\end{eqnarray*}
\begin{center}
  \begin{tabular}{|c||c|c|c|c|c|c|} \hline   
   &$C_1$&$C_2$&$C_3$&$C_4$&$C_5$&Remarks \\ \hline
\hline
   Inverse&$C_1$&$C_2$&$C_3$&$C_4$&$C_5$& \\ \hline
   Cardinality&$1$&$6$&$8$&$6$&$3$&\\ \hline
   Order&$1$&$2$&$3$&$4$&$2$& \\ \hline
\hline
   $\chi_1$ &$1$&$1$&$1$&$1$&$1$& It is a trivial. \\ \hline
   $\chi_2$ &$1$ &$-1$&$1$&$-1$&$1$&It has order 2. \\ \hline
   $\chi_3$ &$3$ & $1$& $0$& $-1$&$-1$ &\\ \hline
   $\chi_4$ &$3$ &$-1$& $0$ &$ 1$ &$-1$ &$\chi_4=\chi_2\chi_3.$ \\ \hline
   $\chi_5$ &$2$ &$0$ &$-1$ &$ 0$ &$ 2$  &$\chi_5\in F(C_1\cup C_5).$\\ \hline
  \end{tabular}
\end{center}

\begin{itemize}

\item About a character $\chi_3$

\begin{tabular}{|c||c|c|c|c|c|}\hline
               &$C_1$&$C_2$&$C_3$\\ \hline\hline
    $\chi_3$ &$3$ & $1$& $0$\\ \hline
   $\lambda_t(\chi_3)$ &$(1+t)^3$&$(1+t)(1-t^2)$&$1+t^3$ \\ \hline
   $S_t(\chi_3)$ &$\dfrac{1}{(1-t)^3}$ & $\dfrac{1}{(1-t)(1-t^2)}$ & $\dfrac{1}{1-t^3}$\\ \hline
\end{tabular}

\begin{tabular}{|c||c|c|c|c|c|}\hline
               &$C_4$&$C_5$\\ \hline\hline
    $\chi_3$ & $-1$&$-1$ \\ \hline
   $\lambda_t(\chi_3)$ &$(1-t)(1+t^2)$&$(1-t)(1-t^2)$ \\ \hline
   $S_t(\chi_3)$ & $\dfrac{1}{(1+t)(1+t^2)}$&$\dfrac{1}{(1+t)(1-t^2)}$ \\ \hline
\end{tabular}

\begin{eqnarray*}
\lambda_t(\chi_3)&=&\chi_1+\chi_3t+\chi_4t^2+\chi_2t^3. \\
S_t(\chi_3)&=&\dfrac{1}{(1-t^2)(1-t^3)(1-t^4)}\chi_1+\dfrac{t^6}{(1-t^2)(1-t^3)(1-t^4)}\chi_2\\
&+&\dfrac{t}{(1-t)(1-t^2)(1-t^4)}\chi_3+\dfrac{t^3}{(1-t)(1-t^2)(1-t^4)}\chi_4\\&+&\dfrac{t^2}{(1-t^2)^2(1-t^3)}\chi_5 \\
&=&\dfrac{1}{(1-t^2)(1-t^3)(1-t^4)}\Big(\chi_1+t^6\chi_2+(t+t^2+t^3)\chi_3\\
&+&(t^3+t^4+t^5)\chi_4+(t^2+t^4)\chi_5\Big).
\end{eqnarray*} 

\item About a character $\chi_4$

\begin{tabular}{|c||c|c|c|c|c|}\hline
               &$C_1$&$C_2$&$C_3$\\ \hline\hline
    $\chi_4$ &$3$ & $-1$& $0$\\ \hline
   $\lambda_t(\chi_4)$ &$(1+t)^3$&$(1-t)(1-t^2)$&$1+t^3$\\ \hline
   $S_t(\chi_4)$ &$\dfrac{1}{(1-t)^3}$ & $\dfrac{1}{(1+t)(1-t^2)}$ & $\dfrac{1}{1-t^3}$ \\ \hline
\end{tabular}

\begin{tabular}{|c||c|c|c|c|c|}\hline
               &$C_4$&$C_5$\\ \hline\hline
    $\chi_4$ & $1$&$-1$ \\ \hline
   $\lambda_t(\chi_4)$ &$(1+t)(1+t^2)$&$(1-t)(1-t^2)$ \\ \hline
   $S_t(\chi_4)$ & $\dfrac{1}{(1-t)(1+t^2)}$& $\dfrac{1}{(1+t)(1-t^2)}$ \\ \hline
\end{tabular}
\begin{eqnarray*}
\lambda_t(\chi_4)&=&\chi_1+\chi_4t+\chi_4t^2+\chi_1t^3. \\
S_t(\chi_4)&=&\dfrac{1-t^3+t^6}{(1-t^2)(1-t^3)(1-t^4)}\chi_1+\dfrac{t^3}{(1-t^2)(1-t^3)(1-t^4)}\chi_2 \\
&+&\dfrac{t^2}{(1-t)(1-t^2)(1-t^4)}\chi_3+\dfrac{t-t^2+t^3}{(1-t)(1-t^2)(1-t^4)}\chi_4\\
&+&\dfrac{t^2+t^4}{(1-t^2)^2(1-t^3)}\chi_5\\
&=&\dfrac{1}{(1-t^2)(1-t^3)(1-t^4)}\Big((1-t^3+t^6)\chi_1+t^3\chi_2\\
&+&(t^2+t^3+t^4)\chi_3+(t+t^3+t^5)\chi_4+(t^2+t^4)\chi_5\Big).
\end{eqnarray*} 

\item About a character $\chi_5$, which is described in Example \ref{ExampleS4}, it corresponds to 2-dimension irreducible character of $S_3$.

\begin{tabular}{|c||c|c|c|c|c|}\hline
               &$C_1$&$C_2$&$C_3$&$C_4$&$C_5$\\ \hline\hline
    $\chi_3$&$2$&$0$&$-1$&$0$&$2$ \\ \hline
   $\lambda_t(\chi_5)$&$(1+t)^2$&$1-t^2$&$1-t+t^2$&$1-t^2$&$(1+t)^2$ \\ \hline
   $S_t(\chi_5)$ &$\dfrac{1}{(1-t)^2}$ & $\dfrac{1}{1-t^2}$ & $\dfrac{1}{1+t+t^2}$& $\dfrac{1}{1-t^2}$& $\dfrac{1}{(1-t)^2}$ \\ \hline
\end{tabular}

\begin{eqnarray*}
\lambda_t(\chi_5)&=&\chi_1+\chi_5t+\chi_2t^2. \\
S_t(\chi_5)&=&\dfrac{1}{(1-t^2)(1-t^3)}\chi_1+\dfrac{t^3}{(1-t^2)(1-t^3)}\chi_2+\dfrac{t}{(1-t)(1-t^3)}\chi_5 \\
&=&\dfrac{1}{(1-t^2)(1-t^3)}\Big(\chi_1+t^3\chi_2+(t+t^2)\chi_5\Big)
\end{eqnarray*}
\end{itemize}

\subsection{An alternating group $A_5$}\label{resultA5}
An alternating group $A_5$ is one of simple non-commutative group. Conjugate classes and the character table of $A_5$ is the following statement and table where $\eta$ is the primitive $5$-th root of unity. And we let numbers $a$ and $b$ be  $a:=\eta+\eta^{-1}=(-1+\sqrt{5})/2$ and $\ b:=\eta^2+\eta^{-2}=(-1-\sqrt{5})/2$. It is satisfied $a+b=ab=-1$ (\cite{bu-2}\ p.288).
\begin{eqnarray*}
(1)\in C_1,\ (1\ 2)(3\ 4)\in C_2,\ (1\ 2\ 3)\in C_3,\ (1\ 2\ 3\ 4\ 5)\in C_4,\ (1\ 3\ 2\ 4\ 5)\in C_5.
\end{eqnarray*}

\begin{center}
  \begin{tabular}{|c||c|c|c|c|c|c|} \hline
   &$C_1$&$C_2$&$C_3$&$C_4$&$C_5$&Remarks \\ \hline
\hline
   Inverse&$C_1$&$C_2$&$C_3$&$C_4$&$C_5$& \\ \hline
   Cardinality&$1$&$15$&$20$&$12$&$12$& \\ \hline
   Order&$1$&$2$&$3$&$5$&$5$& \\ \hline
\hline
   $\chi_1$ &$1$&$1$&$1$&$1$&$1$&It is a trivial.\\ \hline
   $\chi_2$ &$4$&$0$&$1$&$-1$&$-1$&\\ \hline
   $\chi_3$ &$5$&$1$&$-1$&$0$&$0$& \\ \hline
   $\chi_4$ &$3$ &$-1$& $0$ &$-b$&$-a$& \\ \hline
   $\chi_5$ &$3$ &$-1$ &$0$ &$-a$&$-b$& \\ \hline
  \end{tabular}
\end{center}

\begin{itemize}
\item About $\chi_2$

\begin{tabular}{|c||c|c|c|c|c|}\hline
               &$C_1$&$C_2$&$C_3$\\ \hline\hline
   $\chi_2$ &$4$& $0$ &$1$\\ \hline
   $\lambda_t(\chi_2)$ &$(1+t)^4$ &$(1-t^2)^2$&$(1+t)(1+t^3)$ \\ \hline
   $S_t(\chi_2)$ &$\dfrac{1}{(1-t)^4}$ & $\dfrac{1}{(1-t^2)^2}$ & $\dfrac{1}{(1-t)(1-t^3)}$\\ \hline
\end{tabular}

\begin{tabular}{|c||c|c|c|c|c|}\hline
               &$C_4$&$C_5$\\ \hline\hline
   $\chi_2$ &$-1$ &$-1$\\ \hline
   $\lambda_t(\chi_2)$ &$1-t+t^2-t^3+t^4$&$1-t+t^2-t^3+t^4$ \\ \hline
   $S_t(\chi_2)$ &$\dfrac{1}{1+t+t^2+t^3+t^4}$& $\dfrac{1}{1+t+t^2+t^3+t^4}$ \\ \hline
\end{tabular}
\begin{eqnarray*}
\lambda_t(\chi_2)&=&\chi_1+\chi_2t+(\chi_4+\chi_5)t^2+\chi_2t^3+\chi_1t^4. \\
S_t(\chi_2)&=&\dfrac{1-t^2+t^4-t^6+t^8}{(1-t^2)^2(1-t^3)(1-t^5)}\chi_1\\
&+&\dfrac{t+t^6}{(1-t)(1-t^2)(1-t^3)(1-t^5)}\chi_2+\dfrac{t^2}{(1-t)(1-t^2)^2(1-t^3)}\chi_3\\
&+&\dfrac{t^3}{(1-t)(1-t^2)^2(1-t^5)}\chi_4+\dfrac{t^3}{(1-t)(1-t^2)^2(1-t^5)}\chi_5\\
&=&\dfrac{1}{(1-t^2)^2(1-t^3)(1-t^5)}\Big((1-t^2+t^4-t^6+t^8)\chi_1\\
&+&(t+t^2+t^6+t^7)\chi_2+(t^2+t^3+t^4+t^5+t^6)\chi_3\\
&+&(t^3+t^4+t^5)\chi_4+(t^3+t^4+t^5)\chi_5\Big).
\end{eqnarray*} 

\item About $\chi_3$

\begin{tabular}{|c||c|c|c|c|c|}\hline
               &$C_1$&$C_2$&$C_3$\\ \hline\hline
   $\chi_3$ &$5$& $1$ &$-1$\\ \hline
   $\lambda_t(\chi_3)$ &$(1+t)^5$ &$(1+t)(1-t^2)^2$&$(1-t+t^2)(1+t^3)$ \\ \hline
   $S_t(\chi_3)$ &$\dfrac{1}{(1-t)^5}$ & $\dfrac{1}{(1-t)(1-t^2)^2}$ & $\dfrac{1}{(1+t+t^2)(1-t^3)}$ \\ \hline
\end{tabular}

\begin{tabular}{|c||c|c|c|c|c|}\hline
               &$C_4$&$C_5$\\ \hline\hline
   $\chi_3$  &$0$ &$0$\\ \hline
   $\lambda_t(\chi_3)$ &$1+t^5$&$1+t^5$ \\ \hline
   $S_t(\chi_3)$ & $\dfrac{1}{1-t^5}$& $\dfrac{1}{1-t^5}$ \\ \hline
\end{tabular}
\begin{eqnarray*}
\lambda_t(\chi_3)&=&\chi_1+\chi_3t+(\chi_2+\chi_4+\chi_5)t^2+(\chi_2+\chi_4+\chi_5)t^3+\chi_3t^4+\chi_1t^5.\\
S_t(\chi_3)&=&\dfrac{1-t^2+t^4+t^5+t^6-t^8+t^{10}}{(1-t^2)^2(1-t^3)^2(1-t^5)}\chi_1+\dfrac{t^2+t^3-t^4+t^5+t^6}{(1-t)^2(1-t^3)^2(1-t^5)}\chi_2\\
&+&\dfrac{t+t^2-t^3+t^4+t^5}{(1-t)(1-t^2)^2(1-t^3)^2}\chi_3 
+\dfrac{t^3}{(1-t)^2(1-t^2)^2(1-t^5)}\chi_4\\
&+&\dfrac{t^3}{(1-t)^2(1-t^2)^2(1-t^5)}\chi_5 \\
&=&\dfrac{1}{(1-t^2)^2(1-t^3)^2(1-t^5)}\Big((1-t^2+t^4+t^5+t^6-t^8+t^{10})\chi_1\\ &+&(t^2+3t^3+2t^4+2t^6+3t^7+t^8)\chi_2 \\
&+&(t+2t^2+t^3+2t^4+3t^5+2t^6+t^7+2t^8+t^9)\chi_3\\
&+&(t^3+2t^4+3t^5+2t^6+t^7)\chi_4 \\
&+&(t^3+2t^4+3t^5+2t^6+t^7)\chi_5\Big).
\end{eqnarray*} 

\item About $\chi_4$

\begin{tabular}{|c||c|c|c|}\hline
               &$C_1$&$C_2$&$C_3$\\ \hline\hline
   $\chi_4$ &$3$& $-1$ &$0$\\ \hline
   $\lambda_t(\chi_4)$ &$(1+t)^3$ &$(1-t)(1-t^2)$&$1+t^3$\\ \hline
   $S_t(\chi_4)$ &$\dfrac{1}{(1-t)^3}$ & $\dfrac{1}{(1+t)(1-t^2)}$ & $\dfrac{1}{(1-t^3)}$ \\ \hline
\end{tabular}

\begin{tabular}{|c||c|c|c|c|c|}\hline
               &$C_4$&$C_5$\\ \hline\hline
   $\chi_4$ &$-b$ &$-a$\\ \hline
   $\lambda_t(\chi_4)$ &$1-bt-bt^2+t^3$&$1-at-at^2+t^3$ \\ \hline
   $S_t(\chi_4)$ & $\dfrac{1}{1+bt-bt^2-t^3}$& $\dfrac{1}{1+at-at^2-t^3}$ \\ \hline
\end{tabular}
\begin{eqnarray*}
\lambda_t(\chi_4)&=&\chi_1+\chi_4t+\chi_4t^2+\chi_1t^3.\\
S_t(\chi_4)&=&\dfrac{1+t-t^3-t^4-t^5+t^7+t^8}{(1+t)(1-t^2)(1-t^3)(1-t^5)}\chi_1+\dfrac{t^3}{(1-t)(1-t^3)(1-t^5)}\chi_2\\
&+&\dfrac{t^2}{(1-t^2)^2(1-t^3)}\chi_3 
+\dfrac{t-t^3+t^5}{(1-t^2)^2(1-t^5)}\chi_4+\dfrac{t^3}{(1-t^2)^2(1-t^5)}\chi_5 \\
&=&\dfrac{1}{(1+t)(1-t^2)(1-t^3)(1-t^5)}\Big((1+t-t^3-t^4-t^5+t^7+t^8)\chi_1 \\
&+&(t^3+2t^4+t^5)\chi_2+(t^2+t^3+t^4+t^5+t^6)\chi_3\\
&+&(t+t^2-t^4+t^6)\chi_4+(t^3+t^4+t^5)\chi_5\Big).
\end{eqnarray*} 

\item About $\chi_5$

\begin{tabular}{|c||c|c|c|}\hline
               &$C_1$&$C_2$&$C_3$\\ \hline\hline
   $\chi_5$ &$3$& $-1$ &$0$\\ \hline
   $\lambda_t(\chi_5)$ &$(1+t)^3$ &$(1+t)(1-t^2)$&$1+t^3$ \\ \hline
   $S_t(\chi_5)$ &$\dfrac{1}{(1-t)^3}$ & $\dfrac{1}{(1-t)(1-t^2)}$ & $\dfrac{1}{(1-t^3)}$ \\ \hline
\end{tabular}

\begin{tabular}{|c||c|c|}\hline
               &$C_4$&$C_5$\\ \hline\hline
   $\chi_5$ &$-a$ &$-b$\\ \hline
   $\lambda_t(\chi_5)$ &$1-at-at^2+t^3$&$1-bt-bt^2+t^3$ \\ \hline
   $S_t(\chi_5)$ & $\dfrac{1}{1+at-at^2-t^3}$& $\dfrac{1}{1+bt-bt^2-t^3}$ \\ \hline
\end{tabular}

\begin{eqnarray*}
\lambda_t(\chi_5)&=&\chi_1+\chi_5t+\chi_5t^2+\chi_1t^3.\\
S_t(\chi_5)&=&\dfrac{1+t-t^3-t^4-t^5+t^7+t^8}{(1+t)(1-t^2)(1-t^3)(1-t^5)}\chi_1+\dfrac{t^3}{(1-t)(1-t^3)(1-t^5)}\chi_2\\
&+&\dfrac{t^2}{(1-t^2)^2(1-t^3)}\chi_3 
+\dfrac{t^3}{(1-t^2)^2(1-t^5)}\chi_4\\&+&\dfrac{1-t^3+t^5}{(1-t^2)^2(1-t^5)}\chi_5 \\
&=&\dfrac{1}{(1+t)(1-t^2)(1-t^3)(1-t^5)}\Big((1+t-t^3-t^4-t^5+t^7+t^8)\chi_1 \\ &+&(t^3+2t^4+t^5)\chi_2+(t^2+t^3+t^4+t^5+t^6)\chi_3\\&+&(t^3+t^4+t^5)\chi_4+(t+t^2-t^4+t^6)\chi_5\Big).
\end{eqnarray*} 
\end{itemize}

\subsection{The dihedral group $D_{2n}$}\label{resultD2n}
We define $D_{2n}$ by the group that is generated by two elements $a$ and $b$, which is satisfied the relation $a^n=b^2=e$ and $bab^{-1}=a^{-1}$ for each $n\in\mathbb{N}$. It is a non-commutative group if $n\neq 1$. The group $D_{2n}$ is called the dihedral group.

Conjugate classes and the character table of $D_{2n}$ is different in the case of $n$ is even and $n$ is odd. Now, let $\eta$be the primitive $n$-th root of unity.

\begin{itemize}
\item If $n$ is odd, conjugate classes and the character table of $D_{2n}$ is the following statement and table.
\[ C_i=\{a^i,\ a^{-i}\}\ (i=0,1,\dots,\frac{n-1}{2}),\ C':=\{ ba^i\ |\ i=0,1,\dots,n-1\}. \]
\begin{center}
  \begin{tabular}{|c||c|c|c|} \hline
   &$C_i$&$C'$&Remarks \\ \hline
\hline
   Inverse&$C_i$&$C'$& \\ \hline
   Cardinality&$2$&$n$&The case of $C_0$ is $1$.\\ \hline
   Order&$\frac{n}{(i,n)}$&$2$&The case of $C_0$ is $1$. \\ \hline
\hline
   $\chi_1$ &$1$&$  1$&It is a trivial.\\ \hline
   $\chi_2$ &$1$&$-1$&It has order 2. \\ \hline
   $\tau_k$\ ($k=1,\dots,\frac{n-1}{2}$)&$\eta^{ik}+\eta^{-ik}$&$0$& \\ \hline
\end{tabular}
\end{center}

\item If $n$ is even, conjugate classes and the character table of $D_{2n}$ is the following statement and table.
\begin{eqnarray*}
C_i&:=&\{a^i,\ a^{-i}\}\ (i=0,1,\dots,\frac{n}{2}),\\
C'&:=&\{ ba^{2i}\ |\ i=0,1,\dots,\frac{n}{2}-1\},\\
C''&:=&\{ ba^{2i-1}\ |\ i=0,1,\dots,\frac{n}{2}-1\}.
\end{eqnarray*}
\begin{center}
  \begin{tabular}{|c||c|c|c|c|} \hline
   &$C_i$&$C'$&$C''$&Remarks \\ \hline
\hline
  Inverse&$C_i$&$C'$&$C''$& \\ \hline
  Cardinality&$2$&$\frac{n}{2}$&$\frac{n}{2}$&The case of $C_0$ or \\
&&&& $C_{\frac{n}{2}}$ are $1$.\\ \hline
   Order&$\frac{n}{(i,n)}$&$2$&$2$&The case of $C_0$ is $1$. \\ \hline   
\hline
   $\chi_1$ &$1$&$  1$&$ 1$&It is a trivial.\\ \hline
   $\chi_2$ &$1$&$-1$&$-1$&It has order 2. \\ \hline
   $\chi_3$ &$(-1)^i$&$1$&$-1$&It has order 2. \\ \hline
   $\chi_4$ &$(-1)^i$&$-1$&$1$&It has order 2. \\ \hline
   $\tau_k$\ ($k=1,\dots,\frac{n}{2}-1$)&$\eta^{ik}+\eta^{-ik}$&$0$&$0$& \\ \hline
\end{tabular}
\end{center}
\end{itemize}

 We calculate the character of symmetric and exterior powers representation of 2-dimension irreducible character $\tau_k\ (k=1,2,\dots,\dfrac{n-1}{2})$. It is different in the case of $n$. However, we define $\tau_k'\in\MapC{D_{2n}}$ by
\begin{eqnarray*}
\tau_k'(a^i):=\eta^{ik}+\eta^{-ik},\ \tau_k'(ba^i):=0\ \ (i=0,1,\dots,n-1)
\end{eqnarray*}
for each integer $k$. Hence we have $\tau_k=\tau_k'$ for any $k=1,\dots,\dfrac{n-1}{2}$ if $n$ is odd, and for any $k=1,\dots,\dfrac{n}{2}-1$ if $n$ is even. Therefore, to consider $\tau_k'$, we can also calculate in the case of both odd and even of $n$.

About any $\tau_k'$, we have
\begin{itemize}
\item $\tau_k'\tau_l'=\tau_{k+l}'+\tau_{k-l}'.$
\item $\tau_0'=\chi_1+\chi_2$, and If $n$ is even, $\tau_{\frac{n}{2}}'=\chi_3+\chi_4. $
\item $\tau_k'=\tau_{-k}'.$
\item If $k\equiv l \pmod n$, then $\tau_k'=\tau_l'. $
\item $\chi_2\tau_k'=\tau_k'.$
\end{itemize}
for any integers $k$ and $l$. Using these properties, we obtain characters of symmetric and exterior powers representation of 2-dimension irreducible representation of $D_{2n}$ from the next Proposition. 

\begin{prop}\label{resultD2n}
The following equalities hold for any integer $k$.
\begin{enumerate}
\item $\lambda^2(\tau_k')=\chi_2.$
\item $S^{2m-1}(\tau_k')=\displaystyle\sum_{i=1}^m \tau_{(2i-1)k}',\ S^{2m}(\tau_k')=\displaystyle\sum_{i=1}^m \tau_{2ik}'+\chi_1\ (m\in\mathbb{N}).$
\end{enumerate}
\end{prop}
\begin{proof}
$(1)$ Since $\lambda^2(\tau_k')$ has one-dimemsion, it suffices to calculate that  $\lambda^2(\tau_k')(ba^i)$ for any $i=0,1,\dots,n-1$. Using\ (\ref{Formula1}), we have
\begin{eqnarray}\label{dihedral}
\lambda^2(\tau_k')(ba^i)=\dfrac{1}{2}(\tau_k'(ba^i)^2-\psi^2(\tau_k')(ba^i))=\dfrac{1}{2}(0^2-2)=-1.
\end{eqnarray}
Hence, we obtain $\lambda^2(\tau'_k)=\chi_2$ from character table.

$(2)$ The result of $(1)$ and $(\ref{Formula2})$ give the following equation.
\[ S^m(\tau_k')=\tau_k' S^{m-1} (\tau_k')-\chi_2 S^{m-2}(\tau_k')\ (m\geq 2)\]
To prove $(2)$, we use the induction by $m\in\mathbb{N}$. It is clear that $S^1(\tau_k')$ equals $\tau_k'$, and putting $m=2$ in (\ref{dihedral}), we have
\[ S^2(\tau_k')=\tau_k' S^{1} (\tau_k')-\chi_2 S^{0}(\tau_k')=(\tau_k')^2-\chi_2=\tau_{2k}'+\tau_0'-\chi_2=\tau_{2k}'+\chi_1 \]
Hence, we prove $(2)$ when $m=1,2$ Suppose by induction that $(2)$ holds when $m\in\mathbb{N}$, we have
\begin{eqnarray*}
S^{2m+1}(\tau_k')&=&\tau_k' S^{2m} (\tau_k')-\chi_2 S^{2m-1}(\tau_k')\\
&=&\tau_k' \Big(\Big(\sum_{i=1}^m \tau_{2ik}' \Big)+\chi_1\Big)-\chi_2\Big(\sum_{i=1}^m \tau_{(2i-1)k}' \Big)\\
&=&\Big(\sum_{i=1}^m (\tau_{(2i+1)k}'+\tau_{(2i-1)k}')\Big)+\tau_k'-\Big(\sum_{i=1}^m \tau_{(2i-1)k}' \Big)=\sum_{i=1}^{m+1}\tau_{(2i-1)k}' \\
S^{2m+2}(\tau_k')&=&\tau_k' S^{2m+1} (\tau_k')-\chi_2 S^{2m}(\tau_k')\\
&=&\tau_k' \Big(\sum_{i=1}^{m+1} \tau_{(2i-1)k}' \Big)-\chi_2\Big(\Big(\sum_{i=1}^m \tau_{2ik}'\Big)+\chi_1 \Big)\\
&=&\Big(\sum_{i=1}^{m+1} (\tau_{2ik}'+\tau_{(2i-2)k}')\Big)-\sum_{i=1}^m \tau_{2ik}' -\chi_2=\sum_{i=1}^{m+1}\tau_{2ik}' +\chi_1.
\end{eqnarray*}
Hence, we have $(2)$ when $m+1$. This finishes the induction and the proof.
\end{proof}

\subsection{The generalized quaternion group $Q_{4n}$}\label{resultQ4n}
We define $Q_{4n}$ by the group that is generated by two elements $a$ and $b$, which is satisfied the relation $a^{2n}=e, b^2=a^n$ and $bab^{-1}=a^{-1}$ for each $n\geq 2$. The group $Q_{4n}$ is called the generalized quaternion group.

Conjugate classes and the character table of $Q_{2n}$ are following statements and table where $\eta$ is a primitive $2n$-th root of unity.
\begin{eqnarray*}
C_i&:=&\{a^i,\ a^{-i}\}\ (i=0,1,\dots,n),\\
C'&:=&\{ ba^{2i}\ |\ i=0,1,\dots,n-1\},\\
C''&:=&\{ ba^{2i-1}\ |\ i=0,1,\dots,n-1\}.
\end{eqnarray*}
\begin{itemize}
\item There are four character which have one-dimension. however, These two of the four depends on even or odd of $n$. Its commutator subgroup $D(Q_{4n})$, which is the subgroup generated by the elements $xyx^{-1}y^{-1}$ for $x,y\in Q_{4n}$, is the n-cycle group which is generated an element $a^2$ and the quotient group $Q_{4n}/D(Q_{4n})$ is a group with cardinality 4. The reason is that if $n$ is even, then the quotient group $Q_{4n}/D(Q_{4n})$ is isomorphic to the direct sum of two 2-cycle group since an element $b^2=a^n$ belongs to  $D(Q_{4n})$, and if $n$ is odd, then $Q_{4n}/D(Q_{4n})$ is isomorphic to the 4-cycle group since an element $b^2=a^n$ do not belong to $D(Q_{4n})$.
\begin{center}
\begin{tabular}{|c||c|c|c|c|c|} \hline
   &$C_{2i}$&$C_{2i-1}$&$C'$&$C''$&Remarks \\ \hline
\hline
   Inverse&$C_{2i}$&$C_{2i-1}$&$C'$&$C''$&\\ \hline
   Cardinality&$2$&$2$&$\frac{n}{2}$&$\frac{n}{2}$&The case of $C_0$ and $C_n$ is $1$.\\ \hline
   Order&$\frac{n}{(2i,2n)}$&$\frac{n}{(2i-1,2n)}$&$2$&$2$&The case of $C_0$ is $1$. \\ \hline
\hline
   $\chi_1$ &$1$ &$1$&$  1$&$ 1$&It is a trivial.\\ \hline
   $\chi_2$ &$1$ &$1$&$-1$&$-1$&It has order 2 \\ \hline
   $\chi_3'$ &$1$&$-1$&$1$&$-1$&It is occurred when $n$ is even.\\
 &&&&& It has order 2. \\ \hline
   $\chi_4'$ &$1$&$-1$&$-1$&$1$&It is occurred when $n$ is even. \\ 
 &&&&& It has order 2. \\ \hline
   $\chi_3''$ &$1$&$-1$&$i$&$-i$&It is occurred when $n$ is odd. \\ 
 &&&&& It has order 4. \\ \hline
   $\chi_4''$ &$1$&$-1$&$-i$&$i$&It is occurred when $n$ is odd. \\  &&&&& It has order 4. \\ \hline
\end{tabular}
\end{center}
\item In addition, characters of two-dimension representation, $\tau_k\ (k=1,2,\dots,n-1)$ are induced character of the following character
\[ \zeta_k:\langle a \rangle\rightarrow\mathbb{C},\ \zeta_k(a^i):=\eta^{ik}\ (i=0,1,\dots,2n-1). \]
where $\langle a \rangle$ is a subgroup of $Q_{4n}$, which is a $2n$-cycle group generated by an element $a$ of $Q_{4n}$. The value of these character is the following table.
\begin{center}
\begin{tabular}{|c||c|c|}\hline
 &$C_i$&$C',\ C''$\\ \hline\hline
 $\tau_k\ (k=1,2,\dots,n-1)$&$\eta^{ik}+\eta^{-ik}$&$0$\\ \hline
\end{tabular}
\end{center}
\end{itemize}

we calculate about the character $\tau_k$. As in the case of the dihedral group, let $\tau_k'\in\MapC{Q_{4n}}$ be the following equalities for any integer $k$.
\begin{eqnarray*}
\tau_k'(a^i):=\eta^{ik}+\eta^{-ik},\ \tau_k'(ba^i):=0\ \ (i=0,1,\dots,2n-1).
\end{eqnarray*}
Hence, we have $\tau_k=\tau_k' $ for any $k=1,2,\dots, n-1$.

About any $\tau_k'$, we have
\begin{itemize}
\item $\tau_k'\tau_l'=\tau_{k+l}'+\tau_{k-l}'.$
\item $\tau_0'=\chi_1+\chi_2,\ \tau_n'=\begin{cases}\chi_3'+\chi_4' & ($if $n$ is odd$.), \\ \chi_3''+\chi_4'' & ($if $n$ is even$.).\end{cases}$
\item $\tau_k'=\tau_{-k}'.$
\item $k\equiv l\pmod {2n}\Longrightarrow \tau_k'=\tau_l'. $
\item $\chi_2\tau_k'=\tau_k'.$
\end{itemize}
for any integers $k$ and $l$. Using these properties, we obtain characters of symmetric and exterior powers representation of 2-dimension irreducible representation of $Q_{4n}$ from the next Proposition.
\begin{prop}
The following equalities about an element of $\MapC{Q_{4n}}$ hold for any integer $k$
\begin{enumerate}
\item $\lambda^2(\tau_k')=\chi_2^{k+1}.$\\
\item $S^{2m-1}(\tau_k')=\displaystyle\sum_{i=1}^m \tau_{(2i-1)k}',\ S^{2m}(\tau_k')=\displaystyle\sum_{i=1}^m \tau_{2ik}'+(\chi_2)^{mk}\ \ (m\in\mathbb{N}).$
\end{enumerate}
\end{prop}

\begin{proof} Plan of proof is almost the same as (1).

$(1)$ Since $\lambda^2(\tau_k')$ has one-dimension. Using by (\ref{Formula1}), we have
\begin{eqnarray*}
\lambda^2(\tau_k')(a^i)&=&\dfrac{1}{2}(\tau_k'(a^i)^2-\psi^2(\tau_k')(a^i))=\dfrac{1}{2}(\eta^{2ki}+2+\eta^{-2ki}-\eta^{2ki}-\eta^{-2ki})=1,\\
 \lambda^2(\tau_k')(ba^i)&=&\dfrac{1}{2}(\tau_k'(ba^i)^2-\psi^2(\tau_k')(ba^i))=\dfrac{1}{2}(0^2-\tau_k(a^n))=(-1)^{k+1} 
\end{eqnarray*}
Hence, we obtain $\lambda^2(\tau'_k)=\chi_2^{k+1}$ from character table.

$(2)$ The result of $(1)$ and $(\ref{Formula2})$ give the following equation.
\begin{eqnarray}\label{quaternion}
S^m(\tau_k')=\tau_k' S^{m-1} (\tau_k')-\chi_2^{k+1} S^{m-2}(\tau_k')\ (m\geq 2)
\end{eqnarray}
To prove $(2)$, we use the induction by $m\in\mathbb{N}$. It is clear that $S^1(\tau_k')$ equals $\tau_k'$, and putting m=2 in (\ref{quaternion}), we have
\[ S^2(\tau_k')=\tau_k' S^{1} (\tau_k')-\chi_2^{k+1} S^{0}(\tau_k')=(\tau_k')^2-\chi_2^{k+1}=\tau_{2k}'+\tau_0'-\chi_2^{k+1}=\tau_{2k}'+\chi_2^{k+2}=\tau_{2k}'+\chi_2^{k} \]
Hence we prove $(2)$ when $m=1$. Suppose by induction that $(2)$ holds when $m\in\mathbb{N}$, we have
\begin{eqnarray*}
S^{2m+1}(\tau_k')&=&\tau_k' S^{2m} (\tau_k')-\chi_2^{k+1} S^{2m-1}(\tau_k')\\
&=&\tau_k' \Big(\Big(\sum_{i=1}^m \tau_{2ik}' \Big)+\chi_2^{mk}\Big)-\chi_2^{k+1}\Big(\sum_{i=1}^m \tau_{(2i-1)k}' \Big)\\
&=&\Big(\sum_{i=1}^m (\tau_{(2i+1)k}'+\tau_{(2i-1)k}')\Big)+\tau_k'-\Big(\sum_{i=1}^m \tau_{(2i-1)k}' \Big)=\sum_{i=1}^{m+1}\tau_{(2i-1)k}' \\
S^{2m+2}(\tau_k')&=&\tau_k' S^{2m+1} (\tau_k')-\chi_2^{k+1} S^{2m}(\tau_k')\\
&=&\tau_k' \Big(\sum_{i=1}^{m+1} \tau_{(2i-1)k}' \Big)-\chi_2^{k+1}\Big(\Big(\sum_{i=1}^m \tau_{2ik}'\Big)+\chi_2^{mk} \Big)\\
&=&\Big(\sum_{i=1}^{m+1} (\tau_{2ik}'+\tau_{(2i-2)k}')\Big)-\sum_{i=1}^m \tau_{2ik}' -\chi_2^{mk+k+1}\\
&=&\sum_{i=1}^{m+1}\tau_{2ik}' +\chi_2^{mk+k+2}=\sum_{i=1}^{m+1}\tau_{2ik}'+\chi_2^{m(k+1)}
\end{eqnarray*}
Hence, we have $(2)$ when $m+1$. This finishes the induction and the proof.
\end{proof}

\subsection{Heisenberg\ group\ modulo\ $p$\ $H_p$}\label{result}
We let $H_p$ be the following sets.
\[ H_p:=\Bigl\{ \begin{pmatrix} 1 & a & b \\ 0 & 1 & c \\ 0 & 0 & 1 \end{pmatrix}\in GL(3,\mathbb{Z}/p\mathbb{Z})\ |\ a,b,c\in\mathbb{Z}/p\mathbb{Z} \Bigl\} 
\]
where $p$ is a prime number such that $p\neq 2$. It is said to Heisenberg\ group\ modulo\ $p$. 

Conjugate classes and character table of $H_p$ are the following statements and table were $\eta$ is a primitive $p$-th root of unity.
\begin{eqnarray*}
C(h)&:=&\Bigl\{\begin{pmatrix}1&0&h\\0&1&0\\0&0&1\end{pmatrix}\Bigl\}\ \ (h\in\mathbb{Z}/p\mathbb{Z}), \\
C(e,f)&:=&\Bigl\{\begin{pmatrix}1&e&k\\0&1&f\\0&0&1\end{pmatrix}\ \Bigl|\ k\in\mathbb{Z}/p\mathbb{Z}\Bigl\}\ \ (e,f\in\mathbb{Z}/p\mathbb{Z},\ (e,f)\neq(0,0)).
\end{eqnarray*}
\begin{center}
  \begin{tabular}{|c||c|c|c|} \hline
   &$C(h)$&$C(e,f)$&Remarks \\ \hline
\hline
   Inverse&$C(-h)$&$C(-e,-f)$& \\ \hline
   Cardinality&$1$&$p$&\\ \hline
   Order&$p$&$p$&In the case of $C(0)$ is $1$. \\ \hline
\hline
   $\chi_{i,j}$ &$1$&$\eta^{ei+fj}$&$\chi_{0,0}$ is a trivial \\ 
   $ (i,j=0,1,\dots,p-1)$& & &and the other have order $p$.\\ \hline
   $\tau_s\ (s=1,2,\dots,p-1)$ &$p\eta^{sh}$&$0$& \\ \hline
\end{tabular}
\end{center}

The calculate of $\tau_s$ had already described in Example \ref{ExampleHeisenberg}, so we described the result again. we have 
\begin{eqnarray*}
\lambda^n(\tau_s)&=&\begin{cases} \chi_{0,0} & (n=0,p), \\ \dfrac{1}{p}\ \displaystyle\binom{p}{n}\tau_{n'} & (n=1,\dots ,p-1),\\
0 & (n>p).
\end{cases}\\ 
S^n(\tau_s)&=&\begin{cases} \dfrac{1}{p^2}\Big(\displaystyle\binom{p+n-1}{n}+(p^2-1)\Big)\chi_{0,0}\\ +\displaystyle\sum_{i,j\ (i,j)\neq(0,0)}\dfrac{1}{p^2}\Big(\binom{p+n-1}{n}-1\Big)\chi_{i,j} & (p \mid n), \\ \dfrac{1}{p}  \displaystyle{p+n-1}{n}\tau_{n'} & (p\nmid n). 
\end{cases}
\end{eqnarray*}
for any $n=0,1,2,\dots$ where $n'$ is a natural number such that $n'=1,2,\dots,p-1$ and $ns\equiv n' \pmod p$.
\newpage

\section{Future work}

 As we write this paper this time, there is the most important starting point which is \cite{bu-1}\ $p.95$ In there, Kuntson ,who is an author of \cite{bu-1}. calculate about a character of a 2-th exterior power representation of the irreducible representation of $S_3$ which has dimension 2 as a example (See Example\ \ref{ExampleS31}). However,  it is believed that there is not a literature or paper to refer the calculate of the multiplication of the symmetric and exterior representation of finite group since \cite{bu-1}. However, Seeing the result of the regular representation or irreducible representation of Heisenberg group modulo $p$, it is believed that this will help always somewhere.  In the future works, in addition to calculate about other finite groups and representation, it is important to find the result of this paper.  

As a final note, I would like to thank variety of people including Professor Hiroyuki Ochiai.

$
\\$

Graduate School of Mathematics, Kyushu-University, 

Nishi-ku Fukuoka,  819-0395, Japan.

t-tamura@math.kyushu-u.ac.jp

\end{document}